\documentclass[12pt,article]{amsart}
\usepackage[T1]{fontenc}
\usepackage[english]{babel}
\usepackage{enumerate}
\usepackage{mathtools}
\usepackage{esint}
\usepackage{setspace}
\usepackage{mathrsfs}
\usepackage{bbm}
\usepackage{geometry}
\usepackage{amssymb,amsthm,amsmath} 
\usepackage{indentfirst}
\usepackage{xcolor}
\usepackage{graphicx, caption}
\usepackage{pgfplots}
\usepackage{subfig}
\usepackage[countmax]{subfloat}
\renewcommand{\epsilon}{\varepsilon}

\makeatletter

\usepackage{url}									
\usepackage{hyperref}	
\usepackage{cleveref}

\usepackage{soul}
\usepackage{comment}

\newcommand{\numberset}{\mathbb}
\newcommand{\N}{\numberset{N}}
\newcommand{\R}{\numberset{R}}
\newcommand{\C}{\numberset{C}}
\newcommand{\Z}{\numberset{Z}}

\newcommand{\tr}{\mathsf{T}} 
\newcommand{\D}{\mathcal{D}}
\newcommand{\F}{\mathscr{F}}

\DeclareMathOperator{\supp}{supp}
\DeclareMathOperator{\sgn}{sgn}

\newcommand{\frk}[1]{\ensuremath\mathfrak{#1}}

\theoremstyle{remark}
\newtheorem{remark}{Remark}
\newtheorem*{comments}{Comments}

\theoremstyle{definition}
\newtheorem{definition}{Definition}
\newtheorem*{notations}{Notations}

\theoremstyle{theorem}
\newtheorem{theorem}{Theorem}
\newtheorem{prop}{Proposition}
\newtheorem{lemma}{Lemma}

\DeclareMathOperator{\dom}{dom}

\begin{document}

\title[Strichartz estimates for the massless Dirac-Coulomb equation]{STRICHARTZ ESTIMATES FOR THE 2D AND 3D MASSLESS DIRAC-COULOMB EQUATIONS AND APPLICATIONS}

\author{Elena Danesi}
\address{Elena Danesi: 
Dipartimento di Matematica, Universit\'a degli studi di Padova, Via Trieste, 63, 35131 Padova PD, Italy}
\email{edanesi@math.unipd.it}

\maketitle 

\begin{abstract}
In this paper we continue the analysis of the dispersive properties of the 2D and 3D massless Dirac-Coulomb equations that has been started in \cite{CS16} and \cite{CSZ23}. We prove a priori estimates of the solution of the mentioned systems, in particular Strichartz estimates with an additional angular regularity, exploiting the tools developed in the previous works. As an application, we show local well-posedness results for a Dirac-Coulomb equation perturbed with Hartree-type nonlinearities.
\end{abstract}

\section{Introduction}

In this paper we consider the Cauchy problem associated with the massless Dirac equation with an electric \emph{Coulomb potential} in 2 and 3 spatial dimensions. It reads as follows:
\begin{equation}
\label{eq:DC}
\begin{cases}
 i \partial_t u + \D_n u - \frac{\nu}{\lvert x \rvert} u=0, \\
 u(0,x) = u_0(x)
\end{cases}
\end{equation}
where $u(t,x) \colon \R_t \times \R_x^n \to \C^N$, $n \in \{2,3\}$, $N = 2^{\lceil \frac{n}2 \rceil}$, $\nu \in \big [ - \frac{n-1}2, \frac{n-1}2 \big ]$ and $\D_n$ is the Dirac operator on $\R^n$. In order to define the operator we introduce two sets of matrices:
the Pauli matrices $\sigma_j \in \mathcal M_{2\times 2} (\C)$
\[
\sigma_1 = \begin{pmatrix}  0 & 1 \\  1 & 0 \end{pmatrix}, \quad 
\sigma_2 = \begin{pmatrix}  0 &-i \\  i &0 \end{pmatrix}, \quad
\sigma_3 = \begin{pmatrix} 1 & 0 \\ 0 & -1 \end{pmatrix}
\]
and the Dirac matrices $\alpha_j \in \mathcal M_{4\times4}(\C)$
\[
\alpha_j = \begin{pmatrix} 0_{2\times2} & \sigma_j \\ \sigma_j & 0_{2\times2} \end{pmatrix}, \quad j=1,2,3.
\]
Then we can define the Dirac operator on $\R^2$ as
\[
\D_2 = - i \sigma \cdot \nabla_{\R^2} = - i (\sigma_1 \partial_x + \sigma_2 \partial_y)
\]
and on $\R^3$ as
\[
\D_3 = -i \alpha \cdot \nabla_{\R^3} = - i \sum_{j=1}^3 \alpha_j \partial_j.
\]

The massless Dirac equation is widely used to describe physical systems from Quantum Mechanics;  the 3D equation is a model for the dynamics of massless fermions, such as the neutrinos. The 2D equation appears in the study of propagation of waves spectrally concentrated near some singular points on 2-dimensional honeycomb structures. We remark that among the materials which enjoy this structure one finds the \emph{graphene}, a single-layer sheet of hexagonally-arranged carbon atoms, that is attracting a lot of interest in the recent years due to its countless technological applications (see \cite{FW12} for a survey and the references therein). Notice that also with the Coulomb potential, describing interactions between particles, the system remains physically interesting (e.g., non-perfect graphene).\\

The restriction on the parameter $\nu$ comes form the fact that, according to Quantum Mechanics, we should work with self-adjoint operators on $L^2(\R^n;\C^N)$. We recall that $\D_{n} - \frac{\nu}{\lvert x \rvert}$, defined on $C_0^\infty(\R^n \backslash \{0\}; \C^N)$, is essentially self-adjoint\footnote{That is, it admits only one self-adjoint extension, its closure.} with domain $H^1(\R^n; \C^N)$ if and only if $n=2$ and $\nu=0$ or $n=3$ and $\lvert \nu \rvert < \frac{\sqrt 3}2$ 
. If $n=3$ and $\lvert \nu \rvert = \frac{\sqrt3}2$ the Dirac-Coulomb operator is still essentially self-adjoint but with domain contained in $H^{\frac12} (\R^3;\C^4)$. For any other values of $\nu$ there exist infinitely many different self-adjoint extensions. However, it has been shown that for $\nu$ in the range we consider it is possible to define a distinguished (i.\ e.\ ``physically relevant'') self-adjoint extension. In particular, for $\lvert \nu \rvert < \frac{n-1}2$ one can choose the self-adjoint extension with domain contained in $H^{\frac12} (\R^n; \C^N)$ (see the recent works \cite{ELS19}, \cite{Mul16} and the references therein).

\medskip

The free massless Dirac equation can be listed within the class of \emph{dispersive equations}. 
One way to see it is to look at the link between the Dirac equation and a system of decoupled wave equations; by conjugating $i \partial_t + \D_n$ one gets
\begin{equation}
\label{eq:Dirac-wave}
(i \partial_t + \D_n)(i \partial_t - \D_n) = (- \partial_{tt}^2 + \Delta) \mathbbm 1_{N\times N},
\end{equation}
since the sets of matrices $\{\sigma_k \}_{k =1,2}$, $\{\alpha_j \}_{j =1,2,3}$ are chosen in order to satisfy the anti commutation relations
\[
\begin{split}
& \sigma_k \sigma_j + \sigma_j \sigma_k = 2 \delta_{k,j} \mathbbm 1_{2\times2}, \quad k,j=1,2 \\
& \alpha_k \alpha_j + \alpha_j \alpha_k = 2 \delta_{k,j} \mathbbm 1_{4\times 4}, \quad k,j=1,2,3.
\end{split}
\]

In the last years, many tools have been developed in order to quantify the dispersion of a system. Among these we find a priori estimates on the solutions, such as Strichartz estimates and their generalizations, which appear to be fundamental tools in the study, e.g., of local and global well-posedness of nonlinear systems (possibly with rough regularity data). We recall the \emph{Strichartz estimates} for the free Dirac flow:
\[
\lVert e^{it\D_n} u_0  \lVert _{L^p_t L^q_x} \le C \lVert u_0 \rVert_{\dot H^{\frac{n}2 - \frac{n}q - \frac1p}},
\]
for any couple $(p,q)$ satisfying the wave admissibility condition
\begin{equation}
\label{adm_cond_freecase}
p,q \in [2, +\infty],\quad \frac1p \le \frac{n-1}2 \Big ( \frac12 - \frac1q \Big), \quad (p,q,) \ne \big ( \tfrac4{n-1}, \infty \big ), (\infty, \infty),
\end{equation}
where we denote by $e^{it\D_n}$ the propagator for solutions of \eqref{eq:DC} with $\nu = 0$. Thanks to the relation \eqref{eq:Dirac-wave}, these estimates are deduced directly from the one that hold for the wave equation.
\medskip 

It has been observed (see \cite{JWY12} for a survey and the references therein) that one can enlarge the set of admissible couples, requiring some additional regularity in the angular variable. For the free Dirac equation one has the following \emph{generalized Strichartz estimates}
\begin{equation}
\label{eq:stri_wave}
\lVert e^{it\D_n} u_0 \rVert_{L^p_t L^q_{\lvert x \rvert} L^2_\theta (\R \times \R^n)} \le C \lVert u_0 \rVert_{\dot H^{\frac{n}2 - \frac{n}q - \frac1p}(\R^n)}
\end{equation}
for $(p,q)$ satisfying 
\[
p,q \in [2, + \infty], \quad \frac{n-1}q + \frac1p < \frac{n-1}2 \text{ or } (p,q) = (\infty,2), \quad (p,q) \ne (2, \infty), (\infty, \infty).
\]
In presence of potential it is a natural question to ask whether the dispersion of the system, in the sense described above, is preserved. This analysis for the Coulomb case has been started in \cite{CS16}. The authors showed the validity of the following local smoothing type estimates for solutions $u$ of \eqref{eq:DC}
\begin{equation}
\label{eq:local_smo_CS}
\Big \lVert \lvert x \rvert ^{-\alpha} \big \lvert \D_n - \tfrac{ \nu}{\lvert x \rvert} \big \rvert^{\frac12 - \alpha} u \Big \rVert_{L^2_t L^2_x} \le C \lVert u_0 \rVert_{L^2}
\end{equation}
where $\frac12 < \alpha < \sqrt{ k_n^2 - \nu^2} + \frac12$ and $k_n = \frac{n-1}2$. However, the case $\alpha = \frac12$ is excluded, therefore it doesn't allow to deduce Strichartz estimates using the standard Duhamel formulation and the combination of it with the Strichartz estimates for the free flow. Then, in \cite{CSZ23}, the authors proved asymptotic estimates for the generalized eigenfunctions of the Dirac-Coulomb operator on $\R^3$  and, as an application, some generalized Strichartz estimates for the solutions of \eqref{eq:DC} on $\R^3$.

\medskip

The methods used to obtain these results are inspired by the ones developed in \cite{BPST-Z03} and \cite{MZZ13} to investigate dispersive properties of the wave and Schrödinger equations with inverse-square potential. The Coulomb and the inverse-square potentials leave, respectively, the massless Dirac and the wave/Schrödinger equations invariant under their natural scaling, that in the massless Dirac case is  $u_\lambda (t,x) = u ( \lambda^{-1} t, \lambda^{-1} x)$, $\lambda >0$. This kind of invariance prevents, typically, the use of perturbative methods in the study of the effect of these potentials, requiring the development of new tools. Moreover, the Coulomb-type potentials seem to appear as a natural threshold for the validity of global-in-time Strichartz estimates, in their decays at infinity. Indeed, if we look at the literature concerning Strichartz estimates for the Dirac equation perturbed with electromagnetic potentials we find two type of results; in \cite{DF08} the authors consider the equation $i \partial_t u + \D u + V(x) u = 0$ on $\R^3$, where $V(x) = V(x)^*$ is a $4\times4$ complex valued matrix, decaying (slightly) faster at infinity with respect to the Coulomb potential, precisely such that
\[
\lvert V(x) \rvert \le \frac{\delta}{\lvert x \rvert (1 + \lvert \log \rvert x \rvert \rvert)^{\sigma}}, \quad \sigma > 1.
\]
They showed that the dispersion of the system is preserved, i.e.\ the solution $u$ enjoys the same Strichartz estimates, endpoint excluded, as the free one. The same result can be also extended in the 2-dimensional setting. Other results in this direction include e.g. \cite{BDF11}, \cite{EGG18}. Instead, if one consider potential decaying slower at infinity, it is possibile to construct potential such that the associated system is no more dispersive (in the sense described above). More precisely, in \cite{AFG13} the authors consider the magnetic Dirac equation $i \partial_t u + \D_Au=0$, defining the vector field $A$ as 
\[
A(x) = \lvert x 	\rvert^{-\delta} M x, \quad 1 < \delta < 2, \quad M = \begin{pmatrix} 0&1&0\\-1&0&0\\0&0&0 \end{pmatrix}.
\]
Then, the solutions $u$ of the associated Cauchy problems do not satisfy the Strichartz estimates \eqref{eq:stri_wave} for any admissible couple.
 
 \medskip
The aim of this work is two-folded; firstly, we continue the analysis of the dispersion of \eqref{eq:DC}. We extend, compared to the result in \cite{CSZ23}, the set of admissible couples for the validity of generalized Strichartz estimates in 3D, we prove similar estimate for the 2D system and we also provide new local smoothing estimates.
Then, we apply the obtained results to the study of local well-posedness of nonlinear systems. Before we state the results we introduce the following notations that will be used throughout the paper.

\begin{notations}
We denote with $\D_{n,\nu}$ the Dirac-Coulomb operator acting on $\R^n$, that is $\D_{n,\nu} \coloneqq \D_n - \frac{\nu}{\lvert x \rvert}$ for $n=2,3$ and, with an abuse of notation, we will omit the index $n$ when will be clear from the context.\\

We denote with $\dot H^s$, $s \in \R$, the standard homogeneous Sobolev spaces with the norm $\lVert u \rVert_{\dot H^s} = \big \lVert \lvert D \rvert^s u \big \rVert_{L^2}$ where $\lvert D \rvert = \sqrt {- \Delta}$. Instead we use $\dot H^s_{\D_{\nu}}$ to denote the homogeneous Sobolev spaces induced by the action of the Dirac-Coulomb operator, i.e., with norm $\lVert u \rVert_{\dot H^s_{\D_{\nu}}} = \big \lVert \lvert \D_{\nu} \rvert^s u \big \rVert_{L^2}$. \\

We say that $u \in L^2(\R^n; \C^N)$ is \emph{Dirac-radial} if it coincides with his projection on the first partial wave subspaces. Then, we call $u \in L^2(\R^n; \C^N)$ \emph{Dirac-non radial} if it is orthogonal to Dirac-radial functions. We postpone to Subsection \ref{part_wave_decomp} (\Cref{remark:Dirac-rad}) the precise definitions.\\
With an abuse of notation, we use the term function to refer to both scalars and vector-valued functions. Their nature will be clear from the context.
\end{notations}

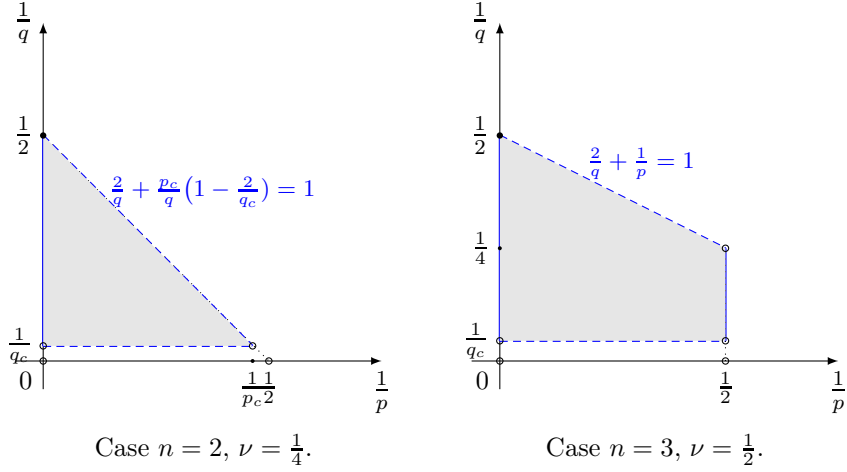
\begin{figure}
\centering
\subfloat[][Case $n=2$, $\nu=\frac14$.]{%
\begin{tikzpicture}[scale=1.5, >=latex]
\draw[->] (-.25,0) -- (3,0) coordinate (x axis); \node[below] at (3,0) {$\frac1p$};
\draw[->] (0,-.25) -- (0,3) coordinate (y axis); \node[left] at (0,3) {$\frac1q$};
\draw[dotted] (0,2) -- (2,0);
\draw[thick, densely dashed, blue] (0,2) -- (1.856, 0.134);
\draw[thick, densely dashed, blue] (0,0.134) -- ((1.856, 0.134);
\draw[thick, blue] (0,2) -- (0,0.134);
\fill[gray!20] (0,0.134) -- (1.856,0.134) -- (0,2) -- cycle;
\fill (1.856,0) circle (0.5pt) node[below] {\footnotesize $\frac1{p_c}$};
\draw (0,0) circle (0.8pt) node[below left] {\footnotesize $0$};
\draw (2,0) circle (0.8pt) node[below] {\footnotesize $\frac12$};
\draw (0,0.134) circle (0.8pt) node[left] {\footnotesize $\frac1{q_c}$};
\draw (1.856, 0.134) circle (0.8pt);
\draw (1.5,1.5) node [blue]{ \scriptsize$\frac2q + \frac{p_c}q \big ( 1 - \frac2{q_c}) =1$};
\fill (0,2) circle (0.8pt) node[left] {$\frac12$};
\end{tikzpicture}
}
\quad
\subfloat[][Case $n=3$, $\nu=\frac12$.]{
\begin{tikzpicture}[scale=1.5, >=latex]
\draw[->] (-.25,0) -- (3,0) coordinate (x axis); \node[below] at (3,0) {$\frac1p$};
\draw[->] (0,-.25) -- (0,3) coordinate (y axis); \node[left] at (0,3) {$\frac1q$};
\draw[thick, blue] (2,0.179) -- (2,1);
\draw[thick, densely dashed, blue] (0,0.179) -- (2, 0.179);
\draw[thick, blue] (0,2) -- (0,0.179);
\draw[thick, densely dashed, blue] (0,2) -- (2,1);
\draw[dotted] (2,0) -- (2,1);
\fill[gray!20] (0,0.179) -- (2,0.179) -- (2,1) -- (0,2) -- cycle;
\draw (0,0) circle (0.8pt) node[below left] {\footnotesize $0$};
\draw (2,0) circle (0.8pt) node[below] {\footnotesize $\frac12$};
\draw (0,0.179) circle (0.8pt) node[left] {\footnotesize $\frac1{q_c}$};
\draw (2, 0.179) circle (0.8pt);
\draw (1.25,1.75) node [blue] {\scriptsize $\frac2q + \frac1p =1$};
\fill (0,1) circle (0.5pt) node [left] {\footnotesize $\frac14$};
\draw (2,1) circle (0.8pt);
\fill (0,2) circle (0.8pt) node[left] {$\frac12$};
\end{tikzpicture}
}
\caption{Generalized Strichartz estimates}
\label{fig:gSe}
\end{figure}

Our first main result concerns Strichartz estimates with an additional angular regularity for solutions of \eqref{eq:DC}, as explained above. As we will see we will have some technical conditions that will force us to slightly restrict the set of admissible $\nu$. Moreover, respect to the free case, we obtain a smaller set of admissibility for the indexes $p,q$ (as shown in \Cref{fig:gSe}). However we can recover the classical range \eqref{adm_cond_freecase} in two cases: if $\nu =0$\footnote{except for the segment corresponding to $q = + \infty$.}or, for all $\nu \in \big [ -\frac{n-1}2, \frac{n-1}2 \big]$, if the initial datum is Dirac-non radial.\footnote{The reason behind this additional restriction on the indexes will appear clearly in the proofs of theorems \ref{thm:Stri_2D} and \ref{thm:Stri_3D}; it is due to the behavior near the origin of the ``first'' generalized eigenfunctions;} We separate the cases $n=2$ and $n=3$ in order to lighten the notations. We have the following

\begin{theorem}[Strichartz estimates 2D]
\label{thm:Stri_2D}
Let $\lvert \nu \rvert < \frac{\sqrt2}3$ and $(p,q)$ such that
\begin{equation}
\label{adm_cond_2d}
p_c < p \le + \infty, \quad 2 \le q < q_c, \quad \frac2q + \frac{p_c}p \bigg( 1 - \frac2{q_c} \bigg) < 1 \text{ or } (p,q) = (\infty,2),
\end{equation}
where $q_c = \frac2{\frac12 - \gamma_{\frac12}}$, $p_c = \frac{\gamma_{\frac12} + \frac12}{\gamma_{\frac12}}$ and $\gamma_{\frac12} = \sqrt{ \frac14 - \nu^2}$. Then, there exists a constant $C > 0$ such that for any $u_0 \in \dot H^s_{\D_{\nu}}(\R^2;\C^2)$ the following Strichartz estimates hold
\begin{equation}
\label{eq:Stri_2D}
\lVert e^{it \D_{\nu}} u_0 \rVert_{L^p_t L^q_{r dr} L^2_{\theta}} \le C \lVert u_0 \rVert_{\dot H_{\D_{\nu}}^s}
\end{equation}
provided $s = 1- \frac1p - \frac2q$.\\
Moreover, if $u_0$ is Dirac-non radial, then the Strichartz estimates hold for all $\lvert \nu \rvert \le \frac12$ and $(p,q)$ satisfying the admissibility condition \eqref{adm_cond_2d} with $(p_c,q_c) = (2, +\infty)$, $q =q_c$ included.
\end{theorem}

\begin{theorem}[Strichartz estimates 3D] 
\label{thm:Stri_3D}
Let $\lvert \nu \rvert < \frac{\sqrt{15}}4$ and  $(p,q)$ such that
\begin{equation}
\label{adm_cond_3d}
2 \le p \le + \infty, \quad 2 \le q < q_c, \quad \frac2q + \frac1p < 1 \text{ or } (p,q) = (\infty,2),
\end{equation}
where $q_c = \frac3{1 - \sqrt{1 - \nu^2}}$. Then, there exists a constant $C > 0$ such that for any $u_0 \in \dot H^s_{\D_{\nu}}(\R^3;\C^4)$, the following Strichartz estimates hold
\begin{equation}
\label{eq:Stri_3D}
\lVert e^{it \D_{\nu}} u_0 \rVert_{L^p_t L^q_{r^2 dr} L^2_{\theta}} \le C \lVert u_0 \rVert_{\dot H_{\D_\nu}^s}
\end{equation}
provided $s= \frac32 - \frac1p - \frac3q$.\\
Moreover, if $u_0$ is Dirac-non radial, then the Strichartz estimates hold for all $\lvert \nu \rvert \le 1$ and $(p,q)$ satisfying the admissibility condition \eqref{adm_cond_3d} with $q =q_c=+\infty$, included. 
\end{theorem}

We describe in \Cref{fig:gSe} the admissible range (region in grey) in the cases $n=2$, $\nu = \frac14$ and $n=3$, $\nu=\frac12$. Notice that if $\nu=0$ we can extend the region up to the segments $\overline{0P}$ where $P=\big(\frac12, 0\big)$, reaching the classical range.

\begin{remark}
Observe that in the 2D case the norm in the R.\ H.\ S.\ is the one induced by the Dirac-Coulomb operator and it is not, in general, equivalent to the standard Sobolev norm. A more detailed discussion is to be found in Subsection \ref{sect:equiv_sob_norm}.
\end{remark}

\begin{remark}
We notice that from the Strichartz estimates \eqref{eq:Stri_2D}, \eqref{eq:Stri_3D} it is possibile to deduce standard Strichartz estimates with an additional loss of angular derivatives on $u_0$. The idea is to combine the estimates above with Sobolev embeddings on the sphere of dimension $n-1$. We denote with $\Lambda_\theta^s$ the angular derivative operator, which is defined in terms of the Laplace-Beltrami operator on $\mathbb S^{n-1}$
\[
\Lambda_\theta^s = (1- \Delta_{\mathbb S^{n-1}} )^{\frac s2}.
\]
This operator does not commute with the Dirac operator $\D_n$. However, it has been observed in \cite{CD12} (see formula (2.45))\footnote{In \cite{CD12} the authors work in the 3D setting. However, one can use the same definition to build a modified operator, with the same properties, also in the 2D setting.}  that it is possibile to define a modified operator $\tilde \Lambda _\theta^s$ commuting with $\D_n$ and such that
\[
\big \lVert \tilde \Lambda_\theta^s f \big \rVert_{L^2_\theta} \simeq \big \lVert \Lambda^s_\theta f \big \rVert_{L^2_\theta}.
\]
Then, by Sobolev embeddings $H^\sigma_\theta (\mathbb S^{n-1}) \hookrightarrow L^q_\theta(\mathbb S^{n-1})$, we get 
\[
\lVert e^{it \D_{\nu}} u_0 \rVert_{L^p_t L^q_x} \lesssim \big \lVert \Lambda_\theta ^\sigma e^{it \D} u_0 \big \rVert_{L^p_t L^q_{r^{n-1} dr}} \simeq \big \lVert e^{it \D} \tilde \Lambda ^\sigma_\theta u_0 \big \rVert_{L^p_t L^q_{r^{n-1} dr}} \lesssim \big \lVert  \tilde \Lambda ^\sigma_\theta u_0 \rVert_{\dot H_s}
\]
where $\sigma = \frac{n-1}2 - \frac{n-1}q$, $ 0<s < \frac{n-1}2$ and $\nu$, $(p,q)$  as in Theorems \ref{thm:Stri_2D}, \ref{thm:Stri_3D}.
\end{remark}

The strategy developed in order to prove the above results is a refinement of the one in \cite{CSZ23}; let us briefly describe it. Firstly, we exploit the radial structure of the Dirac-Coulomb operator, using the partial wave decomposition to reduce the Dirac-Coulomb operator to a differential operator acting only on the radial component; then the ``relativistic'' Hankel transform (see Section \ref{section:rel_Hank_trans} for the definition) allows us to reduce the radial operator to a multiplicative and diagonal one. This gives us an explicit representation for the solutions of \eqref{eq:DC} (see \eqref{formula:decomp_sol}). Then, we use the pointwise estimates of the generalized eigenfunctions to estimate the $L^p_t (\R)L^q_r([R,2R])$-norm of the solution, $R > 0$, on a fixed angular level for frequency localized initial data. Finally we conclude using the orthogonality of the partial waves, scaling argument and dyadic decompositions.\\

With the same tools, we can also show the validity of a new smoothing estimate for solutions of \eqref{eq:DC}.

\begin{theorem}
\label{thm:loc_smo}
Let $u_0 \in L^2 (\R^n; \C^N)$. Then the following estimate holds 
\begin{equation}
\label{camp-morrey}
\sup_{R > 0} R^{- \frac12} \lVert e^{it \D_{\nu}} u_0 \rVert_{L^2_t L^2_{\lvert x \rvert \le R}} \lesssim \lVert u_0 \rVert_{L^2_x}.
\end{equation}
\end{theorem}

\begin{remark}
We observe that the estimate in the Morrey spaces $L^{1,2}(\R^n)$ can be viewed as a limiting case of the local smoothing estimates \eqref{eq:local_smo_CS} as $\alpha \to \frac12$. We mention also that the same estimate has been proved in \cite{BDF11} for solutions of small magnetic Dirac equations with completely different techniques and in \cite{CF17} for the Dirac equation in the Aharonov-Bohm magnetic field. 
\end{remark}

\medskip

As mentioned above, Strichartz estimates can be used in the study of the dynamics of dispersive nonlinear systems. Therefore, as an application of Theorem \ref{thm:Stri_3D} for Dirac-radial functions, we discuss a local well-posedness result for the following 3D nonlinear system  
\begin{equation}
\label{NL:conv}
\begin{cases}
i \partial_t u + \D_{\nu} u = N(u),\\
u(x,0) = u_0(x), 
\end{cases}
\end{equation}
where $\lvert \nu \rvert < \frac{\sqrt 3}2$, $u_0$ is Dirac-radial and the nonlinearity is of the form
\begin{equation}
\label{shape_nonlin}
N(u) = \big ( \omega \ast \langle \beta u,u \rangle\big ) u, \quad \beta = \begin{pmatrix} \mathbbm 1_2 & 0_2\\0_2&- \mathbbm 1_2\end{pmatrix}
\end{equation}
with $\omega$ radially symmetric, i.e., $\omega(x) = \omega(\lvert x \rvert)$, $\omega \in L^p (\R^3)$ for some $p > 1$.

 We have the following
\begin{theorem}
\label{thm:nla1}
Let $\omega$ be a radial function, $\omega \in L^p(\R^3)$, $p \in \big [ \tfrac32, + \infty \big ]$. Let also $u_0 \in \dot H^s(\R^3)$, $s = \frac3{2p}$, be Dirac-radial. Then there exists a positive time $T=T(\lVert u_0 \rVert_{\dot H^s}, \lVert \omega \rVert_{L^p})$ such that \eqref{NL:conv} has a unique solution $u \in C([0,T]; \dot H^s(\R^3))$.
\end{theorem}

\begin{theorem}
\label{thm:nla2}
Let $\omega$ be a radial function, $\omega \in L^p (\R^3)$, $p \in \big ( \frac3{1 + 2 \sqrt{1 - \nu^2}}, + \infty\big )$. Let also $u_0 \in \dot H^s (\R^3)$, $ s = \frac1p$, be Dirac-radial. Then there exists a positive time $T=T(\lVert u_0 \rVert_{\dot H^s}, \lVert \omega \rVert_{L^p})$ such that \eqref{NL:conv} has a unique solution $u \in C([0,T]; \dot H^s(\R^3)) \cap L^{2p}([0,T]; L^{2p'} (\R^3))$, where $p'$ is the conjugate exponent of $p$.
\end{theorem}

\medskip
The choice of such initial datum is twofold motivated. First, we notice that (see \Cref{prop:stri_rad}) if $u_0$ is Dirac-radial then the Strichartz estimates hold in the classical way. Second, the Dirac-radiality of the solution is preserved by the nonlinearity (as observed in \Cref{rmk:radiality}). Then, the proofs of Theorems \ref{thm:nla1} and \ref{thm:nla2} will be based on standard fixed-point arguments on suitable complete metric spaces (see \cite{CO06}, \cite{MT09} and references therein for related results). Observe that in the first one we do not use any kind of Strichartz estimates; they come into play if we want to require less integrability for $\omega$. However, we shall remark that, in both cases, we strongly exploit the equivalence between $\dot H^s$ and $\dot H^s_{\D_{\nu}}$ norms, which holds for every $s \in [0,1]$ if $\lvert \nu \rvert < \frac{\sqrt3}2$. The failing of this equivalence in the general 2D setting (see Subsection \ref{sect:equiv_sob_norm}), prevents us from extending the result on $\R^2$. To conclude, notice that the conditions on $\omega$ are satisfied by the Yukawa potential $V_b(x) = c \frac{e^{-b\lvert x \rvert}}{\lvert x \rvert} \in L^p(\R^3)$ for all $p <3$, $b>0$ and by $\omega(x) = \langle x \rangle^{-\alpha} \in L^p(\R^3)$ for all $p > \frac3{\alpha}$, $\alpha > 0$. It would be interesting to consider the case $\omega = \delta_0$, in order to recover the standard cubic nonlinearity $N(u) = \langle \beta u,u\rangle u$. However, one expects to obtain local well-posedness results in $H^s(\R^3)$ in the subcritical case $s > s_c$. The critical exponent $s_c$ is given by the homogeneity of the Cauchy problem and it can be obtained by scaling arguments. In this case $s_c = 1$. However, for $s>1$ we don't have anymore the equivalence between $\dot H^s$ and $\dot H^s_{\D_{\nu}}$ norms. This would prevent the use of standard tools developed in Sobolev spaces, e.g.\ Sobolev embeddings, and would require additional work to adapt them in the spaces obtained by the action of the Dirac-Coulomb operator. This would be the object of future works.

\medskip

The paper is divided into 4 Sections. In Section \ref{sec:preliminaries} we introduce the necessary preliminaries: the partial wave decomposition, the ``relativistic'' Hankel transform and we recall the estimates for the generalized eigenfuction in 3D, deducing the ones in 2D. We also discuss the relationship between the standard Sobolev norms and the ones induced by the Dirac-Coulomb operator. Section \ref{sec:proof_l_thm} is devoted to the proofs of the a priori estimates, i.e., Theorems \ref{thm:Stri_2D}, \ref{thm:Stri_3D} and \ref{thm:loc_smo}. Lastly, in Section \ref{sec:proof_nl_thm} we prove the nonlinear applications, i.e.\ Theorems \ref{thm:nla1} and \ref{thm:nla2}.

\section{Preliminaries}
\label{sec:preliminaries}

In this section we collect some tools that we will use in the proof of the results presented above. 

\subsection{Partial wave decomposition}
\label{part_wave_decomp}
A crucial aspect of the Dirac-Coulomb operator is that it can be seen as a radial operator with respect to some suitable decomposition, both in dimension 2 and 3. We recall these decompositions, refering to \cite{T92} for all the details. \\
We use spherical coordinates to write 
\[
L^2(\R^n ; \C^N) \cong L^2 ( (0, \infty), r^{n-1}dr ) \otimes L^2 (\mathbb S^{n-1}; \C^N).
\]
Then, we use the partial wave decompositions to define the isomorphisms
\[
\begin{split}
L^2 (\mathbb S^1; \C^2) & \cong \bigoplus_{k \in \Z + \frac12} \frk h^2_k,\\
L^2( \mathbb S^2; \C^4) & \cong \bigoplus_{k \in \Z^*} \bigoplus_{m_k = -\lvert k \rvert + \frac12}^{\lvert k \rvert - \frac12} \frk h^3_{k,m_k}
\end{split}
\]
where each subspace $\frk h^n$ is two-dimensional and it is left invariant under the action of the Dirac-Coulomb operator. The orthonormal basis of each $\frk h^2_k$, $\big \{ \Xi^+_k(\theta), \Xi^-_k (\theta) \big \}$ is expressed in terms of the ``classic Fourier basis'', that is
\begin{equation}
\label{circ_harm}
\Xi^+ _k (\theta) = \frac{1}{ \sqrt{2 \pi}} \begin{pmatrix}  e^{i ( k - 1/2) \theta} \\ 0 \end{pmatrix}, \quad 
\Xi^- _k (\theta) = \frac{1}{ \sqrt{2 \pi}} \begin{pmatrix} 0 \\ e^{i ( k + 1/2) \theta}  \end{pmatrix},
\end{equation}
analogously, the orthonormal basis of each $\frk h^3_{k,m_k}$, $\big \{ \Xi^+_{k, m_k}(\theta_1, \theta_2), \Xi^-_{k,m_k} (\theta_1, \theta_2) \big \}$ is expressed in terms of the standard spherical harmonics $Y^m_l(\theta_1, \theta_2)$, that is 
\begin{equation}
\label{sph_harm}
\Xi^+_{k,m_k} (\theta_1,\theta_2)= \begin{pmatrix} i \Omega^{m_k}_{k} \\ 0_2 \end{pmatrix}, \quad \Xi^-_{k, m_k} = \begin{pmatrix} 0_2 \\ \Omega^{m_k}_{-k} \end{pmatrix}
\end{equation}
where
\[
\Omega_{k,m_k} = \frac1{\sqrt{\lvert 2k + 1 \rvert}} \begin{pmatrix} \sqrt{ \lvert k-m_k+ \tfrac12\rvert} Y^{m_k - \frac12}_{\lvert k + \frac12 \rvert - \frac12} \\ \sgn(-k) \sqrt{\lvert k+ m_k + \tfrac12 \rvert} Y^{m_k+\frac12}_{\lvert k + \frac12\rvert - \frac12} \end{pmatrix}.
\]

We thus have the unitary isomorphisms
\[
L^2 (\R^2; \C^2 ) \cong \bigoplus_{k} L^2 ((0,\infty), r dr) \otimes \frk h^2_k
\]
\[
L^2(\R^3;\C^4) \cong \bigoplus_{k,m_k} L^2((0,\infty), r^2 dr) \otimes \frk h^3_{k,m_k}
\]
given, respectively, by the decompositions 
\begin{equation}
\begin{split}
\label{decomp}
\Phi^2 (x) &= \sum_{k \in \Z + \frac12}  \phi^+_k (r) \,  \Xi^+_k (\theta) + \phi^-_k(r) \, \Xi^-_k (\theta) \\
\Phi^3 (x) &= \sum_{k \in \Z^*} \sum_{m = - \lvert k \rvert + \frac12}^{\lvert k \rvert - \frac12} \psi^+_{k,m_k}(r) \Xi^+_{k,m_k}(\theta_1, \theta_2) + \psi^-_{k,m_k}(r) \Xi^-_{k,m_k} (\theta_1, \theta_2).
\end{split}
\end{equation}
The action of $\D_{n,\nu}$ on each partial wave subspace $L^2((0,\infty), r^{n-1}dr)^2 \otimes \frk h^n$ can be represented by the radial matrices
\begin{equation}
\label{d^n_k}
d_{\nu,k}^n = \begin{pmatrix} - \frac{\nu}r & - (\frac{d}{dr} + \frac{n-1}{2r}) + \frac{k}r \\ \frac{d}{dr} + \frac{n-1}{2r}+ \frac{k}r & - \frac{\nu}r  \end{pmatrix},
\end{equation}
which are well defined on $C_0^\infty ((0, \infty), r^{n-1}dr )^2 \subset L^2((0,\infty), r^{r-1}dr)^2$.\\
Notice that formulas \eqref{d^n_k} only depend on $k$, then, in order to give a unified treatment of the two cases $n=2,3$, in what follows we will maintain only the dependence on the parameter $k$ for $n=3$. Moreover, we will omit the dependence on $n$ of the angular part.
Thus, if $\varphi \in L^2(\R^n; \C^4)$ we will decompose it as
\begin{equation}
\label{form:decomp}
\Phi(x) = \sum_{k \in \mathcal A _n} \varphi_k^+ (r) \Xi^+_k(\theta) + \varphi_k^- (r) \Xi_k^- (\theta) = \sum_{k \in \mathcal A_n} \varphi_k(r) \cdot \Xi_k (\theta)
\end{equation}
where if $n=2$, $\mathcal A _2 = \Z + \frac12$, $\theta \in \mathbb S^1$ and the functions $\Xi_k$ are the ones defined in \eqref{circ_harm}, instead if $n=3$, $\mathcal A_3 = \Z^*$, $\theta \in \mathbb S^2$ and the functions $\Xi_k$ are as in \eqref{sph_harm} (omitting the dependence of $m_k$). With this decomposition, by Stone's theorem, the propagator is given by 
\[
e^{it \D_{\nu}} \Phi(x) = \sum_{k \in \mathcal A_n} \varphi^+_k (t, r) \, \Xi^+_k (\theta) + \varphi_k^- (t,r) \, \Xi^-_k (\theta)
\]
where
\[
\begin{pmatrix} \varphi^+_k (t,r) \\ \varphi^-_k(t,r) \end{pmatrix} =
 e^{it d_{\nu,k}} \begin{pmatrix} \varphi^+_{k} (r) \\ \varphi^-_{k} (r) \end{pmatrix}.
\]
\begin{remark}
\label{remark:Dirac-rad}
Observe that the radial functions (meaning a vector of four radial functions) are contained in the firsts eigenspaces (correspondig to $k=\pm \frac12$ if $n=2$ and to $k = \pm 1$ if $n=3$) but they are not left invariant, in general, by the Dirac operator. In order to consider invariant sets of functions, we call $u \in L^2(\R^n; \C^N)$ \emph{Dirac-radial} if in the decomposition given by \eqref{form:decomp}, for $n=2$ $u^\pm_k(r) = 0$ for all $\lvert k \vert > \frac12$ and for $n=3$ $u^\pm_{k,m_k}(r) = 0$ for all $\lvert k \vert > 1$. On the contrary, $u \in L^2(\R^n; \C^N)$ is \emph{Dirac-non radial} if it is orthogonal to the first partial wave subspaces; more precisely if, in the decomposition given by \eqref{decomp}, for $n=2$, $u^\pm_k(r) = 0$ if $\lvert k \rvert =\frac12$ and, for $n=3$, $u_{k,m_k}^\pm(r)=0$ if $\lvert k \rvert =1$.
\end{remark}


\subsection{Relativistic Hankel transform}
\label{section:rel_Hank_trans}
Once one has decomposed the Dirac-Coulomb operator in a sum of radial operators, the key idea is to looking for an isometry that transforms each radial differential operator into a multiplication operator. This is the rule of the ``relativistic'' Hankel transform, which is built with the generalized eigenfunctions $\psi_{k,\epsilon}$ of $\D_{n,\nu}$. The idea of this construction was borrowed by \cite{BPST-Z03} in which the author considered the Hankel transform. This is built with the Bessel functions that are the generalized eigenfunctions for the radial Schrödinger operator (see e.g.\ \cite{BPST-Z03}, Section 2.1). In this sense the transform we consider can be viewed as a relativistic counterpart of the standard one.
For the sake of completeness we recall in this Subsection the definition and its properties, without proofs. We refer to \cite{CS16} (Section 2.2) for a complete presentation.
\begin{definition}
Let $\Phi \in L^2((0,\infty) ,r^{n-1} dr) \otimes \frk h^n_k$ for some fixed $k$ and let $\varphi(r) = (\varphi_1(r), \varphi_2(r))$ be the vector of its radial coordinates in decomposition \eqref{form:decomp}. \\We define the following integral transform
\[
\mathcal P_k \varphi (\epsilon) = \begin{pmatrix} \mathcal P^+_k \varphi (\epsilon) \\ \mathcal P^-_k \varphi(\epsilon) \end{pmatrix} = \int_0^{+\infty} H_k (\epsilon r) \cdot \varphi (r) r^{n-1} \, dr
\]
where 
\[
H_k (\epsilon r) = \begin{pmatrix} F_{k,\epsilon} (r) & G_{k,\epsilon}(r) \\ F_{k, -\epsilon} (r) & G_{k, -\epsilon} (r) \end{pmatrix} = \begin{pmatrix} \psi_{k,\epsilon} (r) ^{\tr} \\ \psi_{k, -\epsilon} (r) ^{\tr} \end{pmatrix}.
\]
Here
\begin{equation}
\label{fomula:psi_r}
\psi^n_{k,\epsilon} (r) = \begin{pmatrix} F^n_{k,\epsilon} (r) \\ G^n_{k,\epsilon}(r) \end{pmatrix}, 
\end{equation}
represents the vector of radial coordinates of the generalized eigenfunctions. That is, in the notation of \eqref{form:decomp}, $\Psi^n (x) = \sum_{k \in \mathcal A_n} \psi^n_{k,\epsilon} (r) \cdot \Xi_k (\theta)$ ``solves''
\[
\Big ( \D_n - \frac{\nu}{\lvert x \rvert} \Big )\Psi^n = \epsilon \Psi^n, \quad \epsilon >0.
\]
For the sake of completeness we recall the formulas for $F^n_{k,\epsilon}, G^n_{k,\epsilon}$, that are
\begin{equation}
\label{formula:eigen}
\begin{split}
F^n_{k,\epsilon} (r) &= \frac{\sqrt2 \lvert \Gamma (\gamma + 1 + i \nu) \rvert}{\Gamma (2\gamma +1)} e^{\frac{\pi \nu}{2}} (2\epsilon r)^{\gamma - \frac{n-1}2} Re \{ e^{i(\epsilon r + \xi)}  {_1F_1} (\gamma - i \nu, 2 \gamma +1, -2i \epsilon r) \},\\
G^n_{k,\epsilon} (r) &= \frac{i \sqrt2 \lvert \Gamma (\gamma + 1 + i \nu) \rvert}{\Gamma (2\gamma +1)} e^{\frac{\pi \nu}{2}} (2\epsilon r)^{\gamma - \frac{n-1}2} Im \{ e^{i(\epsilon r + \xi)} {_1F_1} (\gamma - i \nu, 2 \gamma +1, -2i \epsilon r) \}
\end{split}
\end{equation}
where ${_1F_1}(a,b,z)$ are confluent hypergeometric functions, $\gamma = \sqrt{k^2 - \nu^2}$ and $e^{-2i \xi} = \frac{\gamma -i\nu}k$ is a phase shift.
\end{definition}   

\begin{remark}
\label{rmk:gen_eig_neg}
We recall that the spectrum of $\D_n$ is purely absolutely continuous and it is the whole real line $\R$.
The formulas in \eqref{formula:eigen} give the generalized eigenstates of the continuous spectrum corresponding to the positive energies $\epsilon >0$. The ones corresponding to negative energies can be obtained using a charge conjugation argument. Then, one gets that 
\[
\psi^n_{k, -\epsilon}(r) = \begin{pmatrix} F^n_{k,-\epsilon} (r) \\ G^n_{k,-\epsilon}(r) \end{pmatrix} =  \begin{pmatrix} \tilde F^n_{-k,\epsilon} (r) \\ \tilde G^n_{-k,\epsilon}(r) \end{pmatrix}
\]
where $\tilde F^n$, $\tilde G^n$ are the functions obtained by \eqref{formula:eigen} by changing the sign of $\nu$.\\
We also want to underlain that the homogeneity of the generalized eigenfunctions $\psi^n_{k,\epsilon}(r)$ with respect to $\epsilon$ and $r$ is the same, in particular we have that $\psi^n_{k,\epsilon}(r)= \psi^n_{k,1} (\epsilon r)$. The lack of this homogeneity in the massive case prevents us from extending the results to that case.
\end{remark}

\begin{prop}
\label{properties_Pk}
The following properties hold:
\begin{enumerate}[i)]
\item $\mathcal P_k$ is an $L^2$-isometry,
\item $(\mathcal P_k d_{\nu,k} \varphi)(\epsilon) = \sigma_3 \epsilon (\mathcal P_k \varphi )(\epsilon)$, $\epsilon > 0$,
\item The inverse transform of $\mathcal P_k$ is given by
\begin{equation}
\label{inverseP_k}
\mathcal P_k^{-1} \varphi (r) = \int_0^{+\infty} H^*_k (\epsilon r) \cdot \varphi(\epsilon) \epsilon^{n-1} d\epsilon
\end{equation}
where
$
H^*_k (\epsilon r) = \begin{pmatrix} 
F_{k,\epsilon} (r) & F_{k,-\epsilon} (r) \\ 
G_{k, \epsilon} (r) & G_{k, -\epsilon} (r) 
\end{pmatrix}.
$
\end{enumerate}
\end{prop}
    
\subsection{Asymptotic estimates of the generalized eigenfunctions}
Other fundamental tools that we will use are some asymptotic estimates of the generalized eigenfunctions $\psi^n_{k}$: these will allow us to obtain an estimate on the $L^q$ norm of such functions on fixed interval, uniformly with respect to the parameter $k$. Notice that this is another substantial difference respect to the cases of the wave equation with an inverse square potential (\cite{MZZ13}) and the Dirac equation with Aharonov-Bohm magnetic potential (see \cite{CF17}). In these cases, indeed, the generalized eigenfunctions are basically Bessel functions for which asymptotic estimates are well known (see \Cref{rmk:bess} below).

\begin{theorem}
\label{thm:point_est}
Given $\nu \in \big [ -\frac{n-1}2, \frac{n-1}2 \big]$ and $k \in \mathcal A_n$, let $\sqrt{k^2-\nu^2}$ and consider the generalized eigenfunctions $\psi_k^n$ of $\D_{\nu,n}$ with eigenvalue $\epsilon=1$ given by formulas \eqref{fomula:psi_r} and \eqref{formula:eigen}. Then there exist positive constants $C,D$ independent of $k,\nu$ such that the following pointwise estimate holds for all $\rho \in \R \backslash \{0\}$:
\begin{equation}
\label{est_eig}
\big \lvert \psi_k^n(\rho) \big \rvert  \le C 
\begin{cases}
( \min \{ \tfrac{\lvert \rho \rvert}2, 1 \} )^{\gamma - \frac{n-1}2} e^{- D \lvert k \rvert}, & \quad 0 < \lvert \rho \rvert \le \max \{ \tfrac{\lvert k \rvert}2, 2 \}, \\
\lvert k \rvert^{- \frac{2n-3}4} \big ( \big \lvert  \lvert k \rvert - \lvert \rho \rvert \big \rvert + \lvert k \rvert ^{\frac13} \big ) ^{- \frac14}  , & \quad \frac{\lvert k \rvert}2 \le \lvert \rho \rvert \le 2 \lvert k \rvert, \\
\lvert \rho \rvert^{- \frac{n-1}2}, & \quad \lvert \rho \rvert  \ge 2 \lvert k \rvert,
\end{cases}
\end{equation}
Moreover, with a possibly larger $C$ and a smaller $D >0$, both independent of $k,\nu$, the following estimate holds
\begin{equation}
\big \lvert (\psi^n_k)'(\rho) \big \rvert \le C 
\begin{cases}
( \min \{ \tfrac{\lvert \rho \rvert}2, 1 \} )^{\gamma - \frac{n+1}2} e^{-D \lvert k \rvert}, & \quad 0 < \lvert \rho \rvert \le \max \{ \tfrac{\lvert k \rvert}2, 2 \}, \\
\lvert k \rvert^{- \frac{2n-3}4} \big ( \big \lvert  \lvert k \rvert - \lvert \rho \rvert \big \rvert + \lvert k \rvert ^{\frac13} \big ) ^{- \frac14}  , & \quad \frac{\lvert k \rvert}2 \le \lvert \rho \rvert \le 2 \lvert k \rvert, \\
\lvert \rho \rvert^{- \frac{n-1}2}, & \quad \lvert \rho \rvert  \ge 2 \lvert k \rvert,
\end{cases}
\end{equation}
\end{theorem}

\begin{proof}
The proof for the case $n=3$ is given in \cite{CSZ23}, Thm.\ 1.1. For the case $n=2$ we observe that, from formulas \eqref{formula:eigen}, we have a relation between the generalized eigenfunctions in dimensions 2 and 3; they differ from a factor $(2\epsilon r)^{\frac12}$ and from the range where the parameter $k$ lives. Then, it is possibile to adapt the proof in \cite{CSZ23}, with minor modifications, to deduce the estimates in dimension 2.
\end{proof}

\begin{remark}
\label{rmk:bess}
Notice that the obtained estimates are the same, for large value of $k$, that holds for the generalized eigenfunctions of the wave equation with inverse square potentials; in fact, in that case the functions are given by
\[
\tilde \psi^n_{\rho} (r) = (\rho r)^{- \frac {n-2}2} J_{\nu(k)} (\rho r)
\]
where $J_\nu$ is the Bessel function of order $\nu$, $\nu(k) = \sqrt{ \mu(k)^2 + a}$, $\mu(k) = \frac{n-2}2 + k$, $k \in \N$ and $a > - \frac{(n-2)^2}4$. The Bessel functions $J_\nu$, for $\nu \ge2$, enjoy the pointwise bounds
\[
\lvert J_\nu (\rho) \rvert \le C
\begin{cases}
e^{- d \nu}, &\quad 0 < \rho < \frac{\nu}2,\\
\nu^{- \frac14} ( \lvert \rho - \nu \rvert + \nu^\frac13 )^{- \frac14}, &\quad \frac{\nu}2 < \rho < 2\nu,\\
\rho^{-\frac12}, &\quad 2\nu < \rho < \infty,
\end{cases}
\]
for some $C$ and $d$ non depending on $\nu$ (see, e.g., \cite{Ste00}).
\end{remark}

\subsection{Equivalence of norms}
\label{sect:equiv_sob_norm}
As stated in the Introduction, in order to prove the local well-posedness of the nonlinear system we exploit the equivalence between the Sobolev norms $H^s$ and the ones induced by the action of the Dirac-Coulomb operator $\D_{n,\nu}$. In this Subsection we recall it and we also add some results in this direction about the 2D setting. \\
For what concerns the 3D case we have the following  

\begin{prop}
\label{prop:equi_norm_3d}
Let $\lvert \nu \rvert < \frac{\sqrt3}2$ and $u \in \dom (\D_{3,\nu})$. Then for every $s \in [0,1]$ there exist two positive constants $C_1$, $C_2$ such that 
\[
C_1 \lVert u \rVert_{\dot H^s(\R^3)} \le \lVert u \rVert_{\dot H^s_{\D_\nu} (\R^3)} \le C_2 \lVert u \rVert_{\dot H^s(\R^3)}.
\] 
\end{prop}

\begin{proof}
We start observing that if $\lvert \nu \rvert < \frac{\sqrt3}2$ then the domain of the self-adjoint extension of $\D_{3,\nu}$ is $H^1(\R^3;\C^4)$ (see \cite{T92} Section 4.3.3). Moreover, for $\nu$ in this range, $\min \{1, \tfrac12 + \sqrt{1 - \nu^2} \} = 1$.  Then, the statement  comes directly from the following Lemma (it is Corollary 1.8 in \cite{FMS19}).

\begin{lemma}
\label{lemma:equiv_3D}
Let $\lvert \nu \rvert \le 1$. For any $f \in C_c^{\infty}(\R^3;\C^4)$ we have
\begin{enumerate}[i)]
\item if $s \in [0, \min \{1, \tfrac12 + \sqrt{1-\nu^2} \} ]$, then
\[
\big \lVert \lvert - \Delta \rvert^{\frac s2} f \big  \rVert_{L^2} \lesssim_{\nu,s} \big \lVert \lvert \D_{\nu} \rvert^s f \big \rVert_{L^2};
\]
\item if $s \in [0,1]$, then
\[
\big \lVert \lvert \D_{\nu} \rvert^s f \big \rVert_{L^2} \lesssim_{\nu,s} \big \lVert \lvert - \Delta \rvert^{\frac s2} f \big  \rVert_{L^2}
\]
\end{enumerate}
\end{lemma}

\end{proof}

In the 2D case we cannot hope to obtain the same results; firstly, we recall that the domain of the distinguished self-adjoint extension of $\D_{2,\nu}$ is not contained in $H^1(\R^2;\C^2)$ (see \cite{MorMu18}, Corollary 16), for any $\nu \ne 0$
. Secondly, the property $ii)$ is proven using the Hardy's inequality, which we know to fail on $\R^2$. \\
However, the distinguished self-adjoint extension is chosen such that its domain is contained in $H^{\frac12}(\R^2;\C^4)$, if $\lvert \nu \rvert < \frac12$. This suggests that at least inequality $i)$ in \Cref{lemma:equiv_3D} could hold if we take $s \in [0, \tfrac12 ]$. In fact, we have the following
\begin{prop}
\label{prop:1equiv_2D}
Let $\lvert \nu \rvert < \frac12$. For any $u \in \dom (\D_{2,\nu})$ and $s \in \big [0, \frac12 \big]$, we have
\[
\big \lVert u \big \rVert_{\dot H^s} \lesssim_{\nu,s} \big \lVert \lvert \D_\nu \rvert^s u \big \rVert_{L^2}.
\]
\end{prop}

\begin{proof}
\label{proof:equi_norm}
From \cite{MorMu18} [Thm 1] we have that for every $\nu \in  \big ( - \frac12, \frac12 \big )$ there exists $C_\nu >0$ such that 
\[
\lvert \D_\nu \rvert \ge C_\nu \sqrt {- \Delta} \otimes \mathbbm{1}_2.
\]
Then, by operator monotonicity of the map $f(t) = t^p$, $p \in [0,1]$ (see \cite{Carl10} [Thm 2.6]), we conclude that
\[
\lvert \D_\nu \rvert^p \ge C_\nu^p \sqrt{- \Delta}^p \otimes \mathbbm1_2,
\]
that is the estimate with $2s = p$.
\end{proof}

For what concerns the reverse inequality, we observe that even if the Hardy's inequality fails in 2D, it still holds for non-radial functions (see \cite{BPST-Z03}, Prop.\ 1 pag 8). Then we have the following

\begin{prop}
Let $\lvert \nu \rvert < \frac12$. For any $u \in \dom(\D_{2,\nu})$, $u$ Dirac-non radial and $s \in [0, \tfrac12]$ we have
\[
\big \lVert \lvert \D_{\nu} \rvert ^s u \big \lVert_{L^2} \lesssim_{\nu,s} \big \lVert \lvert - \Delta \rvert^{\frac s2} u \big \lVert_{L^2}
\]
\end{prop}

\begin{proof}
We first prove that, if $f \in C_c^\infty(\R^2;\C^2)$ and Dirac-non radial, then the Hardy's inequality holds, i.e.
\begin{equation}
\label{eq:hardy_2d}
\big \lVert \lvert x \rvert^{-1} f \big \rVert_{L^2} \le \big \lVert \nabla f \big \rVert_{L^2}.
\end{equation}
We can decompose $f$ as:
\[
f(x) = \sum_{ k \in \Z + \frac12, \, k \ne \frac12} f_k^+(r) \Xi^+_k(\theta) + f_k^- (r) \Xi^-_k(\theta) = \colon f^+(r,\theta) + f^-(r,\theta).
\]
and we recall that $\partial_\theta \Xi^\pm_k (\theta) = \big (k \mp \frac12 \big) \Xi_k^\pm(\theta)$.
Then, 
\[
\int_0^{2\pi} \lvert f^\pm  \rvert ^2 d\theta = \sum_{k \in \Z + \frac12, \, k \ne \frac12} \lvert f^\pm_k \rvert^2 \le \sum_{k \in \Z + \frac12, \, k \ne \frac12} (k \mp \tfrac12)^2 \lvert f^\pm_k \rvert ^2 = \int_0^{2\pi} \lvert \partial_\theta f^\pm \rvert ^2 d\theta
\]
and thus
\[
\big \lVert \lvert x \rvert ^{-1} f \big \rVert_{L^2}^2 \le \int_0^{\infty} \int_0^{2\pi} \frac 1{r^2} \big [ \lvert \partial_\theta f^+ \rvert ^2 + \lvert \partial_\theta f^- \rvert^2 \big ] d\theta r dr \le \big \lVert \nabla f \big \rVert_{L^2}^2.
\]
Hence, by the Cauchy-Schwarz inequality and by \eqref{eq:hardy_2d}, we have 
\[
( \D_\nu)^2 \lesssim (-\Delta) \otimes \mathbbm 1_2.
\]
Then, by monotonicity\footnote{as in the proof of \Cref{prop:1equiv_2D}.}, we get that for any $f \in C_c^\infty(\R^2:\C^2)$, $f$ Dirac-non radial and for any $s \in [0,1]$
\[
\lVert f \rVert_{\dot H^s_{\D_\nu} (\R^2)} \lesssim_{\nu, s} \lVert f \rVert_{\dot H^s (\R^2)}.
\]
We conclude recalling that if $\lvert \nu \rvert < \frac12$ then  $\dom(\D_{2,\nu}) \subset H^\frac12 (\R^2;\C^2)$,  but not in $H^1(\R^2;\C^2)$ (see \cite{Mul16}).
\end{proof}


\section{Proofs of the results}
\label{sec:proof_l_thm}
Before going into the proofs of the above mentioned results, we exploit the tools described in Subsections \ref{part_wave_decomp} and \ref{section:rel_Hank_trans} in order to work with a ``nice'' representation of the solution. Let $u_0 \in L^2 (\R^n; \C^N)$, by \eqref{form:decomp} and property $ii)$ in \Cref{properties_Pk}, we have 
\begin{gather}
\label{formula:decomp_sol}
\notag u_0(x) = \sum_{k \in \mathcal A_n} u_{0,k} (r) \cdot \Xi_k (\theta), \\
e^{it\D_{\nu}} u_0 (x) = \sum_{k \in \mathcal A_n} \mathcal P_k^{-1} \big [ e^{it\rho \sigma_3} ( \mathcal P_k u_{0,k}) (\rho) \big ] (r) \cdot \Xi_k (\theta) 
\end{gather}
where $x = (r, \theta) \in \R^n$. Moreover, thanks to the $L^2$-orthogonal decomposition,  
 \begin{gather}
 \label{formula:norm_l2_id}
\lVert u_0 \rVert_{L^2_x} = \Big ( \sum_{k \in \mathcal A_n} \lVert u_{0,k} (r) \rVert_{L^2_{r^{n-1}dr}} ^2 \Big )^{\frac12}\\
\label{formula:mixed_norm_sol}
\begin{split}
\big \lVert e^{it\D_{\nu}} u_0 \big \rVert_{L^2_t L^2_x} ^2 & = \Big \lVert \sum_{k \in \mathcal A_n} \mathcal P_k^{-1} \big [ e^{it\rho \sigma_3} ( \mathcal P_k u_{0,k}) (\rho) \big ] (r) \cdot \Xi_k (\theta) \Big \rVert_{L^2_t L^2_{r^{n-1}dr} L^2_{\theta}}^2\\
& = \sum_{k \in \mathcal A_n} \Big \lVert \mathcal P_k^{-1} \big [ e^{it\rho \sigma_3} ( \mathcal P_k u_{0,k}) (\rho) \big ] (r) \Big \rVert_{L^2_t L^2_{r^{n-1} dr}}^2.
\end{split}
\end{gather}
More generally, for $p,q \ge 2$, by Minkowski's inequality $\lVert \cdot \rVert_{L^ql^2} \le \lVert \cdot \rVert_{l^2L^q}$ $\forall q \ge2$, we get 
\begin{equation}
\label{formula:decom_norm_mix}
\begin{split}
\big \lVert e^{it\D_{\nu}} u_0 \big \rVert_{L^p_t L^q_{r^{n-1}dr} L^2_{\theta}} 
& = \Big \lVert \mathcal P_k^{-1} \big [ e^{it\rho \sigma_3} ( \mathcal P_k u_{0,k}) (\rho) \big ] (r) \Big \rVert_{L^p_t L^q_{r^{n-1}dr} l^2_k (\mathcal A_n)} \\
& \le \Big ( \sum_{k \in \mathcal A_n} \big \lVert \mathcal P_k^{-1}  \big [ e^{it\rho \sigma_3} ( \mathcal P_k u_{0,k}) (\rho) \big ] (r) \big \rVert^2_{L^p_t L^q_{r^{n-1}dr}} \Big )^{\frac12}.
\end{split}
\end{equation}

\subsection{Proof of \Cref{thm:loc_smo}}

Thanks to \eqref{formula:norm_l2_id} and \eqref{formula:mixed_norm_sol}, it suffices to show that, for any fixed $k \in \mathcal A_n$,
\[
R^{-1} \Big \lVert \mathcal P_k^{-1} \big [ e^{it\rho \sigma_3} ( \mathcal P_k u_{0,k}) (\rho) \big ] (r) \Big \rVert_{L^2_t L^2_{r ^{n-1}dr} ([0,R])}^2 \le C \lVert u_{0,k} (r) \rVert_{L^2_{r^{n-1}dr}}^2
\]
where $C$ is a positive constant independent of $k$ and $R$.

Let $g_k(\rho) = \mathcal P_k u_{0,k}(\rho)$. Then, from \eqref{inverseP_k} and Plancherel's theorem in $t$ (on each component), we get 
\[
\begin{split}
\Big \lVert \mathcal P_k^{-1} \big [ e^{it\rho \sigma_3} ( g_k(\rho) \big ] (r) \Big \rVert_{L^2_t L^2_{r^{n-1}dr}([0,R])}^2
& = \Big \lVert \F^{-1}_{\rho \sigma_3 \to t} \big \{ H^*_k (r\rho) \cdot g_k(\rho) \rho^{n-1} \, \chi_{\R^+} (\rho)  \big \}\Big \rVert_{L^2 _{rdr} ([0,R]) L^2_t} ^2 \\
& \le \int_0^{+ \infty}  \int_0^R \big \lvert H_k^* (r \rho) g_k(\rho) \rho^{n-1} \big \rvert^2 r ^{n-1} dr d\rho \\
& \lesssim \int_0^{+\infty} \int_0^R \big ( \lvert \psi^n_k(r\rho)  \rvert^2 + \lvert \psi^n_k(-r\rho)  \rvert^2 \big) \lvert g_k(\rho) \rvert^2 \rho^{2(n-1)} r^{n-1} dr  d\rho.
\end{split}
\]
We thus need to estimate 
\begin{equation}
\label{psi(rrho)}
\int_0^{+\infty} \bigg ( \int_0^R \lvert \psi^n_k (r\rho) \rvert ^2 r^{n-1} dr \bigg ) \lvert g_k(\rho) \rvert ^2 \rho^{2(n-1)} d\rho.
\end{equation}
 
\begin{lemma}
Let $\psi_k^n$ be the generalized eigenfunction of $\D_{n,\nu}$ with eigenvalue $1$. Then, there exists a positive constant $C$ depending on $n$ but not on $k$ and $\nu$, such that 
\begin{equation}
\label{psi_k}
\frac1R \int_0^R \lvert \psi^n_k (r) \rvert^2 r^{n-1} dr \le C.
\end{equation}
\end{lemma}

\begin{remark}
\label{rmk:est_neg_rho}
We observe that the same estimate holds for $\psi^n_k(-r)$, that is the generalized eigenfunction with eigenvalue $-1$. In fact, as stated in \Cref{rmk:gen_eig_neg}, such eigenfunctions are obtained by $\psi^n_k$ by changing the sign of $k$ and $\nu$. However, all the estimates in the proof will be independent by the sign of $\nu$ and $k$.
\end{remark} 

\begin{proof}
Let $R > 0$. We use \eqref{est_eig} to estimate the integral in \eqref{psi_k}; we need to consider three different cases :
\begin{enumerate}[i)]
\item if $0 < R < \frac{\lvert k \rvert}2$, then we have 
\begin{equation}
\label{Rmink/2}
\begin{split}
\frac1R \int_0^R \lvert \psi^n_k(r) \rvert^2 r^{n-1} dr & \lesssim \frac1R \int_0^R (\min \{ \tfrac r2, 1 \})^{2 \gamma - (n-1)} e^{-2D\lvert k \rvert} r^{n-1} dr\\
& \le \frac1R \int_0^R (\min \{ \tfrac r2, 1 \})^{-(n-1)} e^{-2D\lvert k \rvert} r^{n-1}dr,
\end{split}
\end{equation}
if $\min \{ \tfrac r2, 1 \} = 1$ then
\[
\eqref{Rmink/2} \lesssim \frac1R R^n e^{-2DR} \le C;
\]
otherwise, if $\frac{r}2 \le 1$,
\[
\eqref{Rmink/2} \le \frac1R e^{-2D\lvert k \rvert} 2^{n-1} R \le C; 
\]
\item if $\frac{\lvert k \rvert}2 < R < 2 \lvert k \rvert$, then we split the interval $[0,R] = [0, \lvert k \rvert/2] + [\lvert k \rvert/2, R] = I_1 + I_2$:
\[
\frac1R \int_{I_1} \lvert \psi^n_k(r) \rvert^2 r^{n-1} dr \lesssim \frac1R e^{-2D \lvert k \rvert} {\frac{\lvert k \rvert}2}^n  \le C,
\]
\[
\begin{split}
\frac1R \int_{I_2} \lvert \psi^n_k(r) \rvert^2 r^{n-1} dr & \lesssim R^{-1} \int_{I_2} \lvert k \rvert^{- \frac{2n-3}2} \big ( \big \lvert \lvert k \rvert - r \big \rvert + \lvert k \rvert^{\frac13} \big )^{- \frac12} r^{n-1} dr\\
& \lesssim \lvert k \rvert^{- \frac12} \int_{\frac{\lvert k \rvert}2} ^{2 \lvert k \rvert} \big \lvert \lvert k \rvert - r \big \rvert^{- \frac12} dr \le C;
\end{split}
\]
\item if $R > 2\lvert k \rvert$, then $[0,R] = [0, \lvert k \rvert/2] + [\lvert k \rvert/2, 2\lvert k \rvert] + [2\lvert k \rvert, R] = I_1 + I_2 + I_3$:\\
estimates over $I_1$, $I_2$ as before,
\[
R^{-1} \int_{I_3} \lvert \psi^n_k(r) \rvert^2 r^{n-1} dr \lesssim R^{-1} \int_{2\lvert k \rvert}^R  dr \le C.
\]
\end{enumerate}
\end{proof}  
    
Then, coming back to \eqref{psi(rrho)}, we have, after a change of variable $\xi = r \rho$,
\[
\begin{split}
\eqref{psi(rrho)}  & = \int_0^{+\infty} \bigg ( \int_0^{R\rho} \lvert \psi_k (\xi) \rvert^2 \xi d\xi  \bigg ) \lvert g_k(\rho) \rvert^2 \rho^{n-2} d\rho \\
&\lesssim R \int_0^{+\infty} \lvert g_k(\rho) \rvert^2 \rho^{n-1} d\rho = C R \big \lVert \mathcal P_k u_{0,k}  \big \rVert_{L^2_{\rho d\rho}}^2,
\end{split}
\]
then the claim since $\mathcal P_k$ is an $L^2$-isometry.

\subsection{Proofs of Strichartz estimates}
\label{sec:proof_lin_thm_3}
In order to lighten the notation, in the following we will treat separately the cases $n=2,3$ and we will omit the dependence on $n$ of the generalized eigenfunctions. We start with the 2-dimensional case. 
\begin{proof}[Proof of Theorem \ref{thm:Stri_2D}]
The proof is divided in four steps. The key step is the second one in which we prove the following Lemma. This provides us Strichartz estimates on a fixed radial interval for a frequency localized solution on a fixed angular level. Then the linear estimates in \Cref{thm:Stri_2D} will follow by scaling arguments and interpolation with the standard $L^\infty_t L^2_x$-estimate.
\begin{lemma}
\label{est_loc_pq}
Let $(p,q) \in (2,\infty] \times  [2, + \infty)$ and $k \in \Z + \frac12$. Let $I = \big [\tfrac12,1 \big ]$ and $g_k \in L^2_{\rho d\rho}((0,+\infty))$ such that $\supp(g_k) \subset I$. Then
\[
\Big \lVert \mathcal P_k^{-1} \big [ e^{it \rho \sigma_3} g_k(\rho) \big ](r) \Big \rVert_{L^p_t L^q_{rdr} ([R,2R])} \le C \lVert g_k(\rho) \rVert_{L^2_{\rho d\rho} (I)} \times 
\begin{cases}
R^{\gamma - \frac12 + \frac2q}, & \quad R \le 1 \\ R^{\frac 1q + \delta(p)  ( 1 - \frac2q )}, & \quad R \ge 2,
\end{cases}
\]
where 
\[
\quad \delta(p) = \begin{cases} \frac 1p - \frac12 &\, \text{ if } p < 4,\\ \frac1{4\epsilon} - \frac12 &\, \text{ if } p \ge 4
\end{cases}
\]
and the constant $C$ is independent of $k$, $\nu$, but eventually depends on $p$ and $q$. Moreover, if $q = +\infty$, for all $p >2$ we have
\[
\Big \lVert \mathcal P_k^{-1} \big [ e^{it \rho \sigma_3} g_k(\rho) \big ](r) \Big \rVert_{L^p_t L^\infty_{dr} ([R,2R])} \le C \lVert g_k(\rho) \rVert_{L^2_{\rho d\rho} (I)} \times 
\begin{cases}
R^{\gamma - 1}, & \quad R \le 1 \\ R^{ \delta(p)}, & \quad R \ge 2,
\end{cases}
\]
\end{lemma}

\emph{Step 1}: By relying on the pointwise estimates of \Cref{thm:point_est}, we estimate the $L^q$ norms of the generalized eigenfunctions on a fixed interval of length $R$;
\begin{lemma}
\label{est_psi_R}
\label{est_R-2R}
Let $k \in \Z + \tfrac12$, $\gamma = \sqrt{k^2 - \nu^2}$, $\lvert \nu \rvert \le \frac12$ and $q \in [2, + \infty]$ . Then, the following estimates hold
\[
\big \lVert \psi_k \big \rVert_{L^q ([R,2R])} \le C \times \begin{cases}  R^{\gamma + \frac1q - \frac12}, &\quad \text{ if } R \le 1, \\ 
R^{\beta (q)}, &\quad \text{ if } R \ge 1,
\end{cases}
\]
and
\[
\big \lVert \psi ' _k \big \rVert_{L^q ([R,2R])} \le C \times \begin{cases}
R^{\gamma + \frac 1q - \frac 32}, &\quad \text{ if } R \le 1, \\
R^{\beta(q)}, &\quad \text{ if } R \ge 1,
\end{cases}
\]
where 
\[
\beta(q) = \begin{cases}
\frac1q - \frac 12 &\quad \text{ if } q \in [2,4),\\
\frac1q - \frac13 &\quad \text{ if } q \in [4, + \infty]
\end{cases}
\]
and all the constants are independent of $\gamma, k$, but eventually dependent on $q$.
\end{lemma}

\begin{proof}
We split the proof in two cases:
\begin{enumerate}[i)]
\item $R \le 1$:\\
Let $q \in [2 , + \infty)$; observing that $R \le r \le 2R$ implies $r \le 2$, we have
\[
\int_R^{2R} \lvert \psi_k (r) \rvert^q dr  \lesssim \int_R^{2R} (\min \{\tfrac{r}2, 1 \}) ^{q \gamma - \frac{q}2} e^{-Dq \lvert k \rvert} dr \lesssim \int_R^{2R} r^{q \gamma - \frac{q}2} dr  \lesssim R^{q\gamma - \frac{q}2 +1},
\]
and
\[
\int_R^{2R} \lvert \psi'_k(r) \rvert^q dr \lesssim \int_R^{2R} (\min \{\tfrac{r}2, 1 \}) ^{q \gamma - \frac32 q} e^{-Dq \lvert k \rvert} dr \lesssim \int_R^{2R} r^{q \gamma - \frac32 q} dr  \lesssim  R^{q\gamma - \frac32 q +1}.
\]
Now let $q = + \infty$. Then, as before,
\[
\sup_{r \in [R,2R]} \lvert \psi_k (r) \lvert \lesssim \sup_{r \in [R,2R]} r^{\gamma - \frac12} 2^{- \gamma + \frac12} \le 2^{\frac12} R^{\gamma -\frac12}.
\]
In the same way we estimate $\lVert \psi ' \rVert_{L^\infty}.$\\
\item $R \ge 1$:\\
We write the interval of integration as $[ R, 2R ] = I_1 + I_2 + I_3$, where
\[
I_1 = [ R, 2R ] \cap \big [ 0 , \tfrac{\lvert k \rvert}2 \big ], \quad I_2 = [ R, 2R ] \cap \big [ \tfrac{\lvert k \rvert}2, 2\lvert k \rvert \big ], \quad I_3 = [ R, 2R ] \cap [ 2 \lvert k \rvert, + \infty)
\]
and we estimate each inteval separately.\\
For $I_1$ we can assume $2R \le \lvert k \rvert$, otherwise $I_1 = \emptyset$. 
Let $ q \in [2, + \infty)$, then 
\[
\int_{I_1} \lvert \psi_k (r) \rvert^q dr \lesssim \int_{I_1} (\min \{\tfrac{r}2, 1 \} )^{q \gamma - \frac{q}2} e^{-Dq \lvert k \rvert} dr \le C_{\alpha} R^{- \alpha}, \quad \forall \, \alpha >0;
\]
in fact, if $\min \{ \tfrac{r}2, 1 \} =1$, then
\[
\int_{I_1} \lvert \psi_k (r) \rvert^q dr \lesssim e^{-Dq \lvert k \rvert} \lvert k \rvert^{\alpha} \lvert k \rvert^{-\alpha} \le C_{\alpha} R^{-\alpha},
\]
otherwise $r \le 2$ and, again 
\[
\int_{I_1} \lvert \psi_k (r) \rvert^q dr \lesssim \int_R^{2R} r^{- \frac{q}2} e^{-Dq r} dr \le C_{\alpha} R^{-\alpha};
\]
Let now $q = + \infty$, then
\[
\lvert \psi_k (r) \rvert \lesssim 
\begin{cases}
e^{-DR} & \quad \text{ if } r \ge 2,\\
r^{- \frac12} 2^{\frac12} e^{-DR} &\quad \text{ if } r \le 2
\end{cases}
\]
and we get the claim.\\
For $I_2$ we can assume $\tfrac{\lvert k \rvert}4 \le R \le 2 \lvert k \rvert$, otherwise $I_2 = \emptyset$. Let $q \in [2,4)$, we compute the integral and obtain
\[
\begin{split}
\int_{I_2} \lvert \psi_k (r) \rvert^q dr&  \lesssim  \int_{\frac{\lvert k \rvert}2}^{2\lvert k \rvert} \lvert k \rvert^{-\frac{q}4}\big ( \big \lvert  \lvert k \rvert - r \big \rvert + \lvert k \rvert ^{\frac13} \big ) ^{- \frac q4} dr \\
& = \lvert k \rvert^{1 - \frac q4} \int_{\frac12}^2 \big ( \lvert k \rvert^{\frac13} ( \lvert k \rvert^{\frac23} \lvert 1 - \tau \rvert + 1 ) \big )^{- \frac q4} d\tau \\
& = \lvert k \rvert^{1 - \frac q4 - \frac q{12}} \int_{-\frac12}^2 \big ( \lvert k \rvert^{\frac23} \lvert y \rvert +1 \big )^{-\frac q4} dy = \lvert k \rvert^{1 - \frac q4 - \frac q{12} - \frac23} \int_{- \frac{\lvert k \rvert^{\frac23}}2}^{\lvert k \rvert^{\frac23}} \big ( 1 + \lvert x \rvert )^{-\frac q4} dx \\ 
& = \lvert k \rvert^{\frac 13 - \frac q3} \frac 4{4-q} \bigg [ \big ( 1 + \tfrac{\lvert k \rvert^{\frac23}}2 \big )^{1 - \frac q4} + \big ( 1 + \lvert k \rvert^{\frac23} \big )^{1 - \frac q4} -2 \bigg ] \\
& \lesssim \lvert k \rvert^{\frac13- \frac q3 + \frac23 - \frac q6} = \lvert k \rvert^{1 - \frac q2} \simeq R^{1 - \frac q2} .
\end{split}
\]
For $q \in [4, + \infty)$ we estimate the norm as
 \begin{equation}
 \label{est_I2}
 \int_{I_2} \lvert \psi_k(r) \rvert^q dr  \lesssim \int_{I_2} \lvert k \rvert^{- \frac34 q} \big ( \big \lvert  \lvert k \rvert - r \big \rvert + \lvert k \rvert ^{\frac13} \big ) ^{- \frac q4} r^{\frac q2} dr \lesssim \lvert k \rvert^{- \frac56q} \int_R^{2R} r^{\frac{q}2} \lesssim R^{1 - \frac13 q}.
 \end{equation}
 It remains to estimate the $L^\infty$ norm, for which is enough to observe that, if $r \in I_2$
 \[
 \lvert \psi_k(r) \rvert \lesssim \lvert k \rvert^{- \frac14 - \frac1{12}} \simeq R^{- \frac13}.
 \]
 Lastly, for $I_3$ we get 
 \[
 \int_{I_3} \lvert \psi_k (r) \rvert^q dr \lesssim \int_R^{2R} r^{- \frac q2} dr \lesssim R^{ 1 - \frac q2 }
 \]
 and
 \[
 \sup_{r \in [R,2R]} \lvert \psi_k(r) \rvert \lesssim R^{-\frac12}.
 \]
Similar computations bring to the estimates of $\lVert \psi'_k \rVert_{L^q ([R,2R])}$ for $R \ge 1$.
\end{enumerate}
\end{proof}

\emph{Step 2}: We can now provide an estimate for the localized solution on a fixed angular level $k$ and with the radius $r$ varying on a fixed interval of length $R$.\\\\
In order to lighten the notation, in the following we will write $e^{it\rho}$ instead of $e^{it\sigma_3 \rho}$ since all the estimates are to be understood on every component of the functions with values on $\C^N$. Moreover, we observe that all the estimates hold also for the generalized eigenfunctions corresponding to negative energies.

\begin{proof}[Proof of \Cref{est_loc_pq}]
\begin{enumerate}[i)]
\item $R \le 1$: \\
Let $p \in [2, + \infty]$ and $\Omega$ be an interval. Then by the embedding $ \dot H^{\frac12- \frac1q} (\Omega) \hookrightarrow L^q (\Omega)$, that holds for every $q \in [2, +\infty)$, 
and interpolation we have
\[
\begin{split}
& \bigg \lVert \int_0^{\infty} e^{it\rho} \psi_k (r\rho) g_k(\rho) \rho d\rho \bigg \rVert_{L^p_t L^q_{dr} ([R,2R])}  \\
& \lesssim \bigg \lVert \int_0^{\infty} e^{it\rho} \psi_k (r\rho) g_k(\rho) \rho d\rho \bigg \rVert_{L^p_t \dot H^{\frac12 - \frac1q} ([R,2R])} \\
& \lesssim \bigg \lVert \int_0^{\infty} e^{it\rho} \psi_k (r\rho) g_k(\rho) \rho d\rho \bigg \rVert_{L^p_t L^2_{dr} ([R,2R])}^{\frac12 + \frac1q} \times \bigg \lVert \int_0^{\infty} e^{it\rho} \psi_k (r\rho) g_k(\rho) \rho d\rho \bigg \rVert_{L^p_t \dot H^1 ([R,2R])}^{\frac12- \frac1q} \\
& \lesssim R^{\gamma - \frac12 + \frac1q} \lVert g_k (\rho) \rVert_{L^{p'}_{\rho d\rho}(I)}.
\end{split}
\]
Where the last inequality comes from Minkowski and Hausdorff-Young inequalities and Lemma \ref{est_psi_R}\footnote{Notice that since $\rho \in [\tfrac12, 1]$ then $\rho R \le 1$ and $\rho^{\gamma} \le 1$.}:
\[
\begin{split}
 \bigg \lVert \int_0^{\infty} e^{it\rho} \psi_k(r\rho) g_k(\rho) \rho d\rho \bigg \rVert_{L^p_t L^2_{dr} ([R,2R])} & \lesssim \lVert \psi_k(r\rho) g_k(\rho) \rVert_{L^{p'}_{d\rho}(I) L^2_{dr}([R,2R])} \\
& \lesssim \Big \lVert \lvert g_k(\rho) \rvert  \lVert \psi_k \rVert_{L^2_{dr} [R\rho, 2R\rho]} \Big \rVert _{L^{p'}_{d\rho}(I)}\\
&\lesssim R^{\gamma} \lVert g_k(\rho) \rVert_{L^{p'}_{\rho d\rho}(I)}, 
\end{split}
\]
and
\[
\begin{split}
\bigg \lVert \int_0^{\infty} e^{it\rho} \psi_k '(r\rho) g_k(\rho) \rho d\rho \bigg \rVert_{L^p_t L^2_{dr} ([R,2R])} & \lesssim \lVert \psi '_k(r\rho) g_k(\rho) \rVert_{L^{p'}_{\rho d\rho}(I) L^2_{dr}([R,2R])} \\
&  \lesssim \Big \lVert \lvert g_k(\rho) \rvert \lVert \psi'_k \rVert_{L^2_{dr} [R\rho, 2R\rho]} \Big \rVert _{L^{p'}_{\rho d\rho}(I)}\\
& \lesssim R^{\gamma -1 } \lVert g_k(\rho) \rVert_{L^{p'}_{\rho d\rho}(I)}.
\end{split}
\]
Therefore we obtain
\[
\begin{split}
& \bigg \lVert \int_0^{\infty} e^{it\rho} \psi_k '(r\rho) g_k(\rho) \rho d\rho \bigg \rVert_{L^p_t L^q_{rdr} ([R,2R])} \lesssim R^{\gamma - \frac12+ \frac2q}  \lVert g_k(\rho) \rVert_{L^2_{\rho d\rho}(I)}.
\end{split}
\]
\item $R \ge 2$:\\
Let $p \in [2, + \infty]$. We prove the following estimates
\begin{equation}
\label{LpLinf}
\bigg \lVert \int_0^{\infty} e^{it\rho} \psi_k(r\rho) g_k(\rho) \rho d\rho \bigg \rVert_{L^p_t L^{\infty}_{rdr}([R,2R])} \lesssim R^{\delta (p)} \lVert g_k(\rho) \rVert_{L^2_{\rho d\rho} (I)}
\end{equation}
and 
\begin{equation}
\label{LpL2}
\bigg \lVert \int_0^{\infty} e^{it\rho} \psi_k(r\rho) g_k(\rho) \rho d\rho \bigg \rVert_{L^p_t L^2_{r dr} ([R,2R])} \lesssim R^{\frac12} \lVert g_k (\rho)\rVert_{L^2_{\rho d\rho}(I)},
\end{equation}
then the result will follow by interpolation.
To estimate \eqref{LpLinf} we need to split the cases $p \in [2, 4)$ and $p \in [4,+ \infty]$.\\
If $p < 4 $ we observe that, from Minkowski and Hausdorff-Young inequalities and Lemma \ref{est_psi_R}, we have\footnote{We observe that since $\rho \in [\tfrac12, 1]$ then $R \rho \ge 1$.}
\[
\begin{split}
& \bigg \lVert \int_0^{+\infty} e^{it\rho} \psi_k(r\rho) g_k(\rho) \rho d\rho \bigg \rVert_{L^p_t L^p_{dr}([R,2R])} = \big \lVert \mathcal F_{t \mapsto \rho} \big ( \psi_k(r \rho) g_k(\rho) \rho \chi_{\R^+} (\rho) \big ) \big \rVert_{L^p_{dr} ([R,2R]) L^p_t} \\
& \lesssim \bigg \lVert \big \lVert \psi_k(r\rho) g_k (\rho) \big \rVert _{L^{p'}_{d\rho} (I)}  \bigg \rVert _{L^p_{dr} ([R,2R])}\\
& \lesssim \bigg ( \int_I \lvert g_k (\rho) \rvert^{p'} \bigg ( \int_R^{2R} \lvert \psi_k(r\rho) \rvert^p dr \bigg )^{\frac1p} d\rho \bigg )^{\frac1{p'}} \lesssim R^{\frac1p - \frac12} \lVert g_k(\rho) \rVert_{L^{p'}_{\rho d\rho} (I)}.
\end{split}
\]
The same holds also for $\psi ' _k(r \rho)$. Then, from the embedding $W^{1,p}(\Omega) \hookrightarrow L^\infty(\Omega)$, we get 
\[
\begin{split}
& \bigg \lVert \int_0^{\infty} e^{it\rho} \psi_k(r\rho) g_k(\rho) \rho d\rho \bigg \rVert_{L^p_t L^{\infty}_{rdr}([R,2R])} = \bigg \lVert \int_0^{\infty} e^{it\rho} \psi_k(r\rho) g_k(\rho) \rho d\rho \bigg \rVert_{L^p_t L^{\infty}_{dr}([R,2R])} \\
& \lesssim \bigg ( \bigg \lVert \int_0^{\infty} e^{it\rho} \psi_k(r\rho) g_k(\rho) \rho d\rho \bigg \rVert_{L^p_t L^p_{dr}([R,2R])}^p + \bigg \lVert \int_0^{\infty} e^{it\rho} \psi '_k(r\rho) g_k(\rho) \rho d\rho \bigg \rVert_{L^p_t L^p_{dr}([R,2R])}^p \bigg )^{\frac1p}\\
& \lesssim R^{\frac1p - \frac12} \lVert g_k(\rho) \rVert_{L^2_{\rho d\rho}(I)}.
\end{split}
\]
If $p \ge 4$, we use the embedding $W^{1,\tilde q}(\Omega) \hookrightarrow L^\infty(\Omega)$, $\tilde q = 4\epsilon$ with $\epsilon 
\sim 1$\footnote{This means that the next estimates hold for any $\epsilon \in \big (\frac 12,1 \big)$ but we want $\epsilon$ near 1.} and proceeding as before we get
\[
\bigg \lVert \int_0^{\infty} e^{it\rho} \psi_k(r\rho) g_k(\rho) \rho d\rho \bigg \rVert_{L^p_t L^{\infty}_{rdr}([R,2R])} \\
\lesssim R^{\frac1{4\epsilon} - \frac12} \lVert g_k(\rho) \rVert_{L^2_{\rho d\rho}(I)}.
\]
The estimate \eqref{LpL2} comes easily observing that $R\rho \ge 1$ then from Lemma \ref{est_psi_R} we have
\[
\bigg ( \int_R^{2R} \lvert \psi_k (r \rho) \rvert ^2 r dr \bigg )^{\frac12} \le C R^{\frac12}.
\]
\end{enumerate}
Lastly, we can estimate the $L^\infty_{dr}([R,2R])$ norm as
\[
\begin{split}
\bigg \lVert \int_0^{\infty} e^{it\rho} \psi_k (r\rho) g_k(\rho) \rho d\rho \bigg \rVert_{L^p_t L^\infty_{dr} ([R,2R])} \lesssim R^{\gamma -1} \lVert g_k(\rho) \rVert_{L^2_{\rho d\rho} (I)}, \quad \text{ for } R \le 1,\\
\bigg \lVert \int_0^{\infty} e^{it\rho} \psi_k(r\rho) g_k(\rho) \rho d\rho \bigg \rVert_{L^p_t L^{\infty}_{rdr}([R,2R])} \lesssim R^{\delta (p)} \lVert g_k(\rho) \rVert_{L^2_{\rho d\rho} (I)}, \quad \text{ for } R \ge 2,
\end{split}
\]
where for the former estimate we use the embedding $H^1(\Omega) \hookrightarrow L^\infty(\Omega)$. 
\end{proof}
\emph{Step 3}: We remove assumptions on the localizations  on the solution in order to get the generalized Strichartz estimates with the indexes $p,q$ in the range of admissibility given by the following
\begin{lemma}
\label{true_str_loc_data}
Let 
\[
p \in (2, + \infty], \qquad 6 < q < \frac2{\frac12 - \gamma_{\frac12}}, \qquad \frac1q + \bigg (\frac1p - \frac12\bigg)\bigg(1 - \frac2q\bigg) < 0.
\]
Then
\[
\lVert e^{it \D_{\nu}} u_0 \rVert_{L^p_t L^q _{r dr} L^2_{\theta}} \le C \lVert u_0 \rVert_{\dot H^s_{\D_{\nu}}}
\]
where $s= 1 - \frac1p - \frac2q$ and $\gamma_{\frac12} = \sqrt{\frac14 - \nu^2}$. Moreover, if $u_0$ is Dirac-non radial we can take $q \in (6, + \infty]$. 
\end{lemma}

We remove the assumption of the frequency localization of the solution, to get the Strichartz estimates for $(p,q)$ admissible for \Cref{true_str_loc_data}; 
Thanks to \eqref{formula:decom_norm_mix}, it suffices to show that 
\[
\sum_{k \in \Z + \frac12} \Big \lVert \mathcal P_k^{-1} [ e^{it\rho \sigma_3} (\mathcal P_k u_{0,k}) (\rho) ](r) \Big \rVert_{L^p_t L^q_{rdr}} ^2 \le C \lVert u_0 \rVert_{\dot H^s_{\D_\nu}}^2.
\]
Let $N,R$ be dyadic numbers and $\phi \in C_c^{\infty} \big ( \big [\tfrac12 ,1 \big]\big )$ such that $\sum_{N \in 2^\Z} \phi (\rho N^{-1}) = 1$. Then, for $q < + \infty$, from the embedding $l^2 \hookrightarrow l^q$ and scaling we have the following
\begin{gather}
\begin{split}
& \sum_{k \in \Z + \frac12} \Big \lVert \mathcal P_k^{-1} [e^{-it\rho \sigma_3} (\mathcal P_k u_{0,k})(\rho) ] (r) \Big \rVert_{L^p_t L^q_{r dr}}^2 \\
& = \sum_{k \in \Z + \frac12} \Big \lVert \sum_{N \in 2^{\Z}} \mathcal P_k^{-1} [ e^{-it\rho \sigma_3} (\mathcal P_k u_{0,k})(\rho) \phi(\tfrac{\rho}N) ] (r) \Big \rVert_{L^p_t L^q _{rdr}}^2 \\
& = \sum_{k \in \Z + \frac12} \Big \lVert \Big ( \sum_{R \in 2^{\Z}} \Big \lVert  \sum_{N \in 2^{\Z}} \mathcal P_k^{-1} [ e^{-it\rho \sigma_3} (\mathcal P_k u_{0,k})(\rho) \phi(\tfrac{\rho}N) ] (r) \Big \rVert_{L^q_{rdr}([R,2R])} ^q \Big )^{\frac1q} \Big \rVert_{L^p_t}^2\\ 
& \lesssim  \sum_{k \in \Z + \frac12} \sum_{R \in 2^{\Z}} \Big \lVert  \sum_{N \in 2^{\Z}} \mathcal P_k^{-1} [ e^{-it\rho \sigma_3} (\mathcal P_k u_{0,k})(\rho) \phi(\tfrac{\rho}N) ] (r) \Big \rVert_{L^p_t L^q _{rdr}([ R,2R])}^2 \\
& \lesssim \sum_{k \in \Z + \frac12} \sum_{R \in 2^{\Z}} \Big ( \sum_{N \in 2^{\Z}} \Big \lVert \mathcal P_k^{-1} [ e^{-it\rho \sigma_3} (\mathcal P_k u_{0,k})(\rho) \phi(\tfrac{\rho}N) ] (r) \Big \rVert_{L^p_t L^q_{r dr}([R,2R])} \Big )^2 \\
& = \sum_{k \in \Z + \frac12} \sum_{R \in 2^{\Z}} \Big ( \sum_{N \in 2^{\Z}} N^{2 - \frac1p - \frac2q} \Big \lVert P_k^{-1} [ e^{-it\rho \sigma_3} (\mathcal P_k u_{0,k})(N\rho) \phi(\rho) ] (r) \Big \rVert_{L^p_t L^q_{rdr} ([NR,2NR])} \Big )^2.
\end{split}
\end{gather}
Moreover, from \Cref{est_loc_pq} we can continue the chain of inequalities and we get 
\[
\begin{split}
& \lesssim \sum_{k \in \Z + \frac12} \sum_{R \in 2^{\Z}} \Big ( \sum_{N \in 2^{\Z}} N^{2 - \frac1p - \frac2q} Q(NR) \big \lVert (\mathcal P_k u_{0,k}) (N\rho) \phi(\rho) \big \rVert_{L^2_{\rho d\rho}} \Big )^2 \\
& = \sum_{k \in \Z + \frac12} \sum_{R \in 2^{\Z}} \Big ( \sum_{N \in 2^{\Z}} N^{1 - \frac1p - \frac2q} Q(NR) \big  \lVert (\mathcal P_k u_{0,k}) (\rho) \phi(\tfrac{\rho}N) \big \rVert_{L^2_{\rho d\rho}} \Big )^2
\end{split}
\]
where
\[
Q(NR) =
\begin{cases}
 (NR)^{\gamma - \frac12 + \frac2q} & \quad \text{if } NR \le 1,\\
(NR)^{\frac1q + \delta(p)( 1- \frac2q)} & \quad \text{if } NR \ge 2.
 \end{cases}
 \]
We note that if we take $p,q$ such that
\begin{equation}
\label{formula:cond_gamma}
\gamma - \frac12 + \frac2q > 0, \quad \frac1q + \delta(p) \big( 1 - \frac2q) < 0,
\end{equation}
then
\[
\sup_{R \in 2^{\Z}} \sum_{N \in 2^{\Z}} Q(NR) < + \infty, \quad \sup_{N \in 2^{\Z}} \sum_{R \in 2^{\Z}} Q(NR) < + \infty.
\]
Recall that $\gamma = \sqrt{k^2 - \nu^2} \ge \sqrt2$ if $\lvert k \rvert > \frac12$. Then, the first condition in \eqref{formula:cond_gamma} becomes: $\gamma_{\frac12} - \frac12 + \frac2q > 0$.\\
Let $A_{N,k} = N^{1 - \frac1p- \frac2q} \big \lVert (\mathcal P_k u_{0,k}) (\rho) \phi(\tfrac{\rho}N) \big \rVert_{L^2_{\rho d\rho}}$, we use the Schur test Lemma:
\[
\bigg ( \sum_{R \in 2^{\Z}} \big ( \sum_{N \in 2^{\Z}} Q(NR) A_{N,k} \big )^2 \bigg )^{\frac12} = \sup_{\lVert B_R \rVert_{l^2} \le 1} \sum_{R \in 2^{\Z}} \sum_{N \in 2^{\Z}} Q(NR) A_{N,k} B_R
\]
which is bounded by
\[
\begin{split}
& \le C \bigg ( \sum_{R \in 2^{\Z}} \sum_{N \in 2^{\Z}} Q(NR) \lvert A_{N,k} \rvert^2 \bigg)^{\frac12} \bigg (  \sum_{R \in 2^{\Z}} \sum_{N \in 2^{\Z}} Q(NR) \lvert B_R \rvert ^2 \bigg )^{\frac12} \\
& \le C \big ( \sup_{R \in 2^{\Z}} \sum_{N \in 2^{\Z}} Q(NR) \sup_{N \in 2^{\Z}} \sum_{R \in 2^{\Z}} Q(NR) \big )^{\frac12} \bigg ( \sum_{N \in 2^{\Z}} \lvert A_{N,k} \rvert^2 \bigg )^{\frac12} \bigg ( \sum_{R \in 2^{\Z}} \lvert B_R \rvert^2 \bigg )^{\frac12} \\
& \le C \bigg ( \sum_{N \in 2^{\Z}} \lvert A_{N,k} \rvert ^2 \bigg )^{\frac12}.
\end{split}
\]
Putting the estimates together, we have obtained
\[
\begin{split}
\lVert e^{it \D_{\nu}} u_0 \rVert^2_{L^p_t L^q_{r dr}L^2_{\theta}} & \le C \sum_{k \in \Z + \frac12} \sum_{R \in 2^{\Z}} \Big ( \sum_{N \in 2^{\Z}} Q(NR) A_{N,k} \Big )^2 \\
& \le C \sum_{k \in \Z + \frac12} \sum_{N \in 2^{\Z}} N^{2 \big (1 - \frac1p - \frac 2q\big)} \big \lVert (\mathcal P_k u_{0,k}) (\rho) \phi(\tfrac{\rho}N) \big \rVert_{L^2_{\rho d\rho}}^2  \\
& \le C \lVert u_0 \rVert _{\dot H ^{1-\frac1p - \frac2q}_{\D_{\nu}}}^2.
\end{split}
\]
Let us now suppose that $u_0$ is Dirac-non radial, that is $\mathcal P_k u_{0,k} = 0$ for $\lvert k \rvert = \frac12$. Then the first condition in \eqref{formula:cond_gamma} is satisfied for all $q \le + \infty$. Then, the previous estimates prove the claim for any $q \in (6, + \infty)$. If $q= + \infty$, we argue as before; we take $N,R$ dyadic numbers and $\phi \in C_c^{\infty} \big ( \big [\tfrac12 ,1 \big]\big )$ such that $\sum_{N \in 2^\Z} \phi (\rho N^{-1}) = 1$. By the immersion $l^2 \hookrightarrow l^\infty$, the Minkowski's inequality and scaling we have 
\begin{equation}
\begin{split}
& \sum_{k \in \Z + \frac12} \Big \lVert \mathcal P_k^{-1} [ e^{it\rho \sigma_3} (\mathcal P_k u_{0,k}) (\rho) ](r) \Big \rVert_{L^p_t L^\infty_{dr}} ^2 \\
& \lesssim \sum_{k \in \Z + \frac12} \Big \lVert \sum_{N \in 2^{\Z}} \mathcal P_k^{-1} [ e^{it\rho \sigma_3} (\mathcal P_k u_{0,k}) (\rho)  \phi(\tfrac{\rho}N) ](r) \Big \rVert_{L^p_t L^\infty_{dr}} ^2 \\
& \lesssim \sum_{k \in \Z + \frac12} \bigg \lVert \sup_{R \in 2^{\Z}} \Big \lVert  \sum_{N \in 2^{\Z}} \mathcal P_k^{-1} [ e^{it\rho \sigma_3} (\mathcal P_k u_{0,k}) (\rho) \phi(\tfrac{\rho}N)](r) \Big \rVert_{L^\infty_{dr}([ R,2R])} \bigg \rVert_{L^p_t}^2\\
&\lesssim  \sum_{k \in \Z + \frac12} \bigg \lVert \bigg ( \sum_{R \in 2^\Z} \Big \lVert \sum_{N \in 2^{\Z}} \mathcal P_k^{-1} [ e^{it\rho \sigma_3} (\mathcal P_k u_{0,k}) (\rho) \phi(\tfrac{\rho}N) ](r) \Big \rVert^2_{L^\infty_{dr}[R,2R]} \bigg )^{\frac12} \bigg \rVert_{L^p_t} ^2\\
& \lesssim \sum_{k \in \Z + \frac12} \sum_{R \in 2^{\Z}} \bigg ( \sum_{N \in 2^{\Z}} \Big \lVert \mathcal P_k^{-1} [ e^{it\rho \sigma_3} (\mathcal P_k u_{0,k}) (\rho) \phi(\tfrac{\rho}N) ](r) \Big \rVert_{L^p_t L^\infty_{dr} ([R,2R])} \bigg )^2 \\
& = \sum_{k \in \Z + \frac12} \sum_{R \in 2^{\Z}} \Big ( \sum_{N \in 2^{\Z}} N^{2 - \frac1p } \Big \lVert P_k^{-1} [ e^{it\rho \sigma_3} (\mathcal P_k u_{0,k})(N\rho) \phi(\rho) ] (r) \Big \rVert_{L^p_t L^\infty_{dr} ([NR,2NR])} \Big )^2.
\end{split}
\end{equation}
for every $p > 2$. Then, we conclude using the Schur test Lemma as before. \\
\begin{figure}
\centering
\begin{tikzpicture}[scale=1.5, >=latex]
\draw[->] (-.25,0) -- (3,0) coordinate (x axis); \node[below] at (3,0) {$\frac1p$};
\draw[->] (0,-.25) -- (0,3) coordinate (y axis); \node[left] at (0,3) {$\frac1q$};
\draw[thick,dotted] (0,2) -- (1.856,0.134);
\draw[thick, densely dashed, blue] (0,0.134) -- ((1.856, 0.134);
\draw[thick, blue] (0,0.667) -- (0,0.134);
\draw[thick, densely dashed, blue] (1,0.667) -- (0,0.667);
\fill (1,0.667) circle (0.8pt);
\fill (1.856,0) circle (0.5pt) node[below] {\footnotesize $\frac1{p_c}$};
\draw (0,0) circle (0.8pt) node[below left] {\footnotesize $0$};
\draw (2,0) circle (0.8pt) node[below] {\footnotesize $\frac12$};
\draw (0,0.134) circle (0.8pt) node[left] {\footnotesize $\frac1{q_c}$};
\draw (0,0.667) circle (0.8pt) node[left] {\footnotesize $\frac1{6}$};
\fill (1,0) circle (0.5pt) node[below] {\footnotesize $\frac14$};
\draw (1.856, 0.134) circle (0.8pt);
\fill (0,2) circle (0.8pt) node[left] {$\frac12$};
\draw (0.15,2.15) node {\tiny{A}};
\draw (0.15,0.27) node {\tiny{B}};
\draw (1.9,0.27) node {\tiny{C}};
\draw (1.25,1.80) node { \scriptsize$\frac2q + \frac{p_c}q \big ( 1 - \frac2{q_c}\big ) =1$};
\draw[thick, densely dashed, color=blue, domain=1:1.856, smooth, variable=\x]   plot (\x,{((-2)*\x + 4)/((-\x) +4)});
\draw (2.8,0.5) node [blue]{ \tiny$\frac1q + \big (\frac1p - \frac12 \big ) \big (1 - \frac2q \big) = 0 $};
\end{tikzpicture}
\caption{case $n=2$, $\nu = \frac14$.} 
\label{tri}
\end{figure}
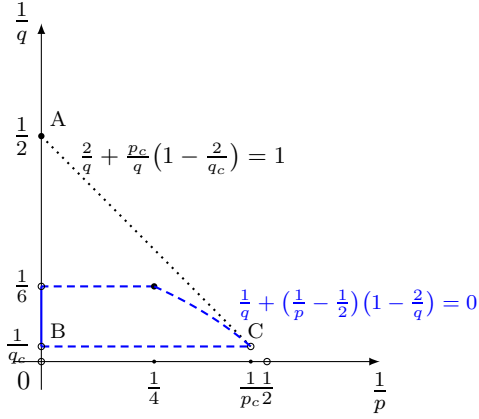

\emph{Step 4}: We observe that the set of admissible exponents is not empty if and only if 
\[
6 < \frac2{\frac12 - \gamma_{\frac12}} \quad \iff \lvert \nu \rvert <  \frac{\sqrt 2}3.
\] 
In this case, we have the validity of the Strichartz estimates in the set with blue boundary described in \Cref{tri}, where we define $q_c \coloneqq \frac2{\frac12 - \gamma_{\frac12}}$ and the corresponding endpoint index satisfying \eqref{formula:cond_gamma} as $p_c$, that is $p_c \coloneqq \frac{\gamma_\frac12 + \frac12}{\gamma_\frac12}$. 
Then, we can interpolate between the estimate in \Cref{true_str_loc_data} with $(p, q_c)$, $ p \in (p_c,+ \infty]$ and the standard estimate
\[
\lVert e^{it\D_{\nu}} u_0 \rVert_{L^\infty_t L^2_{rdr} L^2_\theta} \le C \lVert u_0 \rVert_{L^2}
\]
to reach the triangle $\widehat{ABC}$ (\Cref{tri}). Lastly, the estimates for Dirac-non radial solutions follows from the same argument.
\end{proof}
\subsubsection{Strichartz estimates 3D case}
We now discuss the 3-dimensional case. We retrace the proof of the 2D case. For the sake of brevity we omit the proofs of the first three steps since the computations are basically the same.  \\
\emph{Step 1}:
\begin{lemma}
\label{est_psi_R_3d}
Let $k \in \Z^*$, $\gamma = \sqrt{k^2 - \nu^2}$, $\lvert \nu \rvert \le 1$ and $q \in [2,+\infty]$. Then, the following estimates hold
\[
\big \lVert \psi_k \big \rVert_{L^q ([R,2R])} \le C \times \begin{cases}  R^{\gamma + \frac1q - 1}, &\quad \text{ if } R \le 1, \\ 
R^{\beta (q)}, &\quad \text{ if } R \ge 1,
\end{cases}
\]
and
\[
\big \lVert \psi ' _k \big \rVert_{L^q ([R,2R])} \le C \times \begin{cases}
R^{\gamma + \frac 1q - 2}, &\quad R \le 1, \\
R^{\beta(q)}, &\quad R \ge 1,
\end{cases}
\]
where 
\[
\beta(q) = \begin{cases}
\frac1q - 1&\quad \text{ if } q \in [2,4),\\
\frac1q - \frac56 &\quad \text{ if } q \in [4, + \infty]
\end{cases}
\]
and all the constants are independent of $\gamma, k$, but eventually dependent on $q$.
\end{lemma}
\emph{Step 2}:
\begin{lemma}
Let $(p,q) \in [2, +\infty]\times [2,+\infty)$ and $k \in \Z^*$. Let $I = [\frac12,1]$ and $g_k \in L^2_{\rho^2 d\rho}((0,+\infty))$ such that $\supp(g_k) \subset I$. Then
\[
\Big \lVert \mathcal P_k^{-1} \big [ e^{it \rho \sigma_3} g_k(\rho) \big ](r) \Big \rVert_{L^p_t L^q_{r^2dr} ([R,2R])} \le C \lVert g_k(\rho) \rVert_{L^2_{\rho^2 d\rho} (I)} \times 
\begin{cases}
R^{\gamma - 1 + \frac3q}, & \quad R \le 1 \\ R^{\frac 1q + \delta(p)  ( 1 - \frac2q )}, & \quad R \ge 2,
\end{cases}
\]
where 
\[
\quad \delta(p) = \begin{cases} \frac 1p - 1 &\, \text{ if } p < 4,\\ \frac1{4\epsilon} - 1 &\, \text{ if } p \ge 4.
\end{cases}
\]
and the constant $C$ is independent of $k$, $\nu$, but eventually depends on $p$ and $q$. Moreover, if $q = +\infty$, for all $p \ge 2$ we have
\[
\Big \lVert \mathcal P_k^{-1} \big [ e^{it \rho \sigma_3} g_k(\rho) \big ](r) \Big \rVert_{L^p_t L^\infty_{dr} ([R,2R])} \le C \lVert g_k(\rho) \rVert_{L^2_{\rho d\rho} (I)} \times 
\begin{cases}
R^{\gamma - \frac32}, & \quad R \le 1 \\ R^{ \delta(p)}, & \quad R \ge 2,
\end{cases}
\]
\end{lemma}
Notice that in this case we can take $p=2$ since $\delta(2) < 0$.\\
\emph{Step 3}:
\begin{lemma}
\label{Strlocdata3d}
Let 
\[
p \in [2, + \infty], \qquad \frac{10}3 < q < \frac3{1 - \gamma_1}, \qquad \frac1q + \bigg (\frac1p - 1\bigg)\bigg(1 - \frac2q\bigg) < 0.
\]
Then
\[
\lVert e^{it \D_{\nu}} u_0 \rVert_{L^p_t L^q _{r^2 dr} L^2_{\theta}} \le C \lVert u_0 \rVert_{\dot H^s_{\D_{\nu}}}
\]
where $s= \frac32 - \frac1p - \frac3q$ and $\gamma_1 = \sqrt{1 - \nu^2}$. Moreover, if $u_0$ is Dirac-non radial we can take $q \in \big ( \frac{10}3, + \infty \big]$.
\end{lemma}
\begin{proof}
Notice that it suffices to show that 
\[
\sum_{k \in \Z^*} \Big \lVert \mathcal P_k^{-1} [ e^{it\rho \sigma_3} (\mathcal P_k u_{0,k}) (\rho) ](r) \Big \rVert_{L^p_t L^q_{r^2dr}} ^2 \le C \lVert u_0 \rVert_{\dot H^s_{\D_\nu}}^2.
\]
We proceed as in \Cref{true_str_loc_data}, with suitable modifications.
\end{proof}
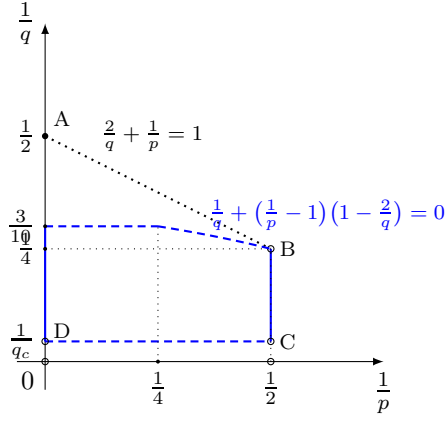
\begin{figure}
\centering
\begin{tikzpicture}[scale=1.5, >=latex]
\draw[->] (-.25,0) -- (3,0) coordinate (x axis); \node[below] at (3,0) {$\frac1p$};
\draw[->] (0,-.25) -- (0,3) coordinate (y axis); \node[left] at (0,3) {$\frac1q$};
\draw[thick, blue] (2,0.179) -- (2,1);
\draw[thick, densely dashed, blue] (0,0.179) -- (2, 0.179);
\draw[thick, blue] (0,1.2) -- (0,0.179);
\draw[thick, dotted] (0,2) -- (2,1);
\draw[dotted] (1,0) -- (1,1.2);
\draw[dotted](0,1) -- (2,1);
\draw[dotted] (2,0) -- (2,1);
\draw[thick,densely dashed, blue] (0,1.2) -- (1, 1.2);
\draw (0,0) circle (0.8pt) node[below left] {\footnotesize $0$};
\draw (2,0) circle (0.8pt) node[below] {\footnotesize $\frac12$};
\draw (0,0.179) circle (0.8pt) node[left] {\footnotesize $\frac1{q_c}$};
\draw (2, 0.179) circle (0.8pt);
\fill (0,1) circle (0.5pt) node [left] {\footnotesize $\frac14$};
\draw (2,1) circle (0.8pt);
\fill (0,1.2) circle (0.5pt) node [left] {\footnotesize $\frac3{10}$};
\fill (0,2) circle (0.8pt) node[left] {\footnotesize $\frac12$};
\fill (1,0) circle (0.5pt) node [below] {\footnotesize $\frac14$};
\draw[thick, densely dashed, color=blue, domain=1:2, smooth, variable=\x]   plot (\x,{((-2)*\x + 8)/((-\x) +6)});
\draw (0.15,2.15) node {\tiny{A}};
\draw (2.15,1) node {\tiny{B}};
\draw (2.15,0.179) node {\tiny{C}};
\draw (0.15, 0.27) node {\tiny D};
\draw (2.5,1.3) node [blue]{ \tiny$\frac1q + \big (\frac1p - 1 \big ) \big (1 - \frac2q \big) = 0 $};
\draw (0.95,2) node { \tiny$\frac2q + \frac1p=1 $};
\end{tikzpicture}
\caption{case $n=3$, $\nu = \frac12$.} 
\label{quad}
\end{figure}

\emph{Step 4}: In the last step we want to combine the previous estimate with the conservation of mass in order to get the Strichartz estimates in \Cref{thm:Stri_3D}. We notice that $q=4$ is admissible if and only if 
\[
4 < \frac3{1- \gamma_1} = q_c \iff \lvert \nu \rvert < \frac{\sqrt{15}}4.
\]
In this case, we have the validity of the Strichartz estimates in the set with blue boundary described in \Cref{quad}, 
Then we interpolate between the estimates in \Cref{Strlocdata3d} with exponent $(p,4)$, $p >2$ and the $L^\infty_t L^2_{r^2 dr} L^2_\theta$-estimate, widening the admissible set up to cover the region $ABCD$ in \Cref{quad}.

\section{Nonlinear application}
\label{sec:proof_nl_thm}
This Section is dedicated to the proof of the well-posedness results for the system \eqref{NL:conv} with Dirac-radial initial datum. We start by proving the classical Strichartz estimates for Dirac-radial functions (\Cref{prop:stri_rad} below) as well as two Lemmas that we will use throughout the proofs of the nonlinear results.
\begin{prop}
\label{prop:stri_rad}
Let $\lvert \nu \rvert < \frac{\sqrt{15}}4$ and  $(p,q)$ such that
\begin{equation}
\label{eq:adm_cond_Dir-rad}
2 \le p \le + \infty, \quad 2 \le q < q_c, \quad \frac2q + \frac1p < 1 \text{ or } (p,q) = (\infty,2),
\end{equation}
where $q_c = \frac3{1 - \sqrt{1 - \nu^2}}$. Then, there exists a constant $C > 0$ such that for any $u_0$ Dirac-radial, $u_0 \in \dot H^s(\R^3;\C^4)$, the following Strichartz estimates hold
\begin{equation}
\lVert e^{it \D_{\nu}} u_0 \rVert_{L^p_t L^q_x} \le C \lVert u_0 \rVert_{\dot H^s}
\end{equation}
provided $s= \frac32 - \frac1p - \frac3q$.
\end{prop}

\begin{proof}
Let $u_0$ be Dirac-radial. From \eqref{formula:decomp_sol} and observing that the functions $\Xi_k$ are bounded, we get
\[
\label{ineq:sum}
\begin{split}
\lVert u \rVert_{L^p_t L^q_x} & = \Big \lVert \sum_{\lvert k \rvert =1} \mathcal P_k^{-1} [ e^{it\rho \sigma_3} (\mathcal P_k u_{0,k}) (\rho) ](r) \cdot \Xi_k (\theta)\Big \rVert_{L^p_t L^q_{r^2dr} L^q_\theta}\\
& \le \sum_{\lvert k \rvert =1} \lVert \Xi_k \rVert_{L^q_\theta} \big \lVert  \mathcal P_k^{-1} [ e^{it\rho \sigma_3} (\mathcal P_k u_{0,k}) (\rho) ](r) \big \rVert_{L^p_t L^q_{r^2dr}} \\
& \lesssim \bigg ( \sum_{\lvert k \rvert =1} \big \lVert \mathcal P_k^{-1} [ e^{it\rho \sigma_3} (\mathcal P_k u_{0,k}) (\rho) ](r) \big \rVert_{L^p_t L^q_{r^2dr}}^2 \bigg )^{\frac12}.
\end{split}
\]
Then we can proceed as in the proof of \Cref{Strlocdata3d} and we conclude by interpolation with the $L^\infty_t L^2_x$ estimate.
\end{proof}

\begin{remark}
We observe that \Cref{prop:stri_rad} holds the same if the initial datum $u_0$ is such that, in the decomposition given by \eqref{formula:decomp_sol}, $u_{0,k} \ne 0$ only for a finite number of $k \in \Z^*$. In fact, inequality \eqref{ineq:sum} remains true if we sum over a finite number of indexes. Moreover, a similar result can be proved for the 2D case.
\end{remark}
\begin{remark}
\label{rmk:radiality}
We observe that $N(u) = \big ( \omega \ast \langle \beta u, u\rangle \big ) u$ preserves the Dirac-radiality of $u$. More precisely, it has been shown in \cite{Cac11} (Lemma 5.5) that if $u$ is a Dirac-radial function, then $\langle \beta u, u\rangle$ is a radial scalar function (in the classical sense). Moreover, we recall that $\omega$ is radial, then $ \omega \ast \langle \beta u, u\rangle$ is a radial function which implies that $N(u)$ is still Dirac-radial.
\end{remark}

\begin{lemma}
\label{lemma:Nu_Hs}
Let $u \in \mathcal S (\R^3;\C^4)$ and $\omega \in L^\gamma (\R^3) \cap L^\alpha (\R^3)$. Then, the following inequality holds
\begin{equation}
\label{Nu_Hs}
\big \lVert \big ( \omega \ast \langle \beta u, u \rangle \big ) u \big \rVert_{\dot H^s} \lesssim \lVert \omega \rVert_{L^{\gamma}} \lVert u \rVert_{\dot H^s} \lVert u \rVert_{L^{\mu_2}} \lVert u \rVert_{L^{p_2}} + \lVert \omega \rVert_{L^{\alpha}} \lVert u \rVert_{\dot H^s} \lVert u \rVert_{L^{2\beta}}^2,
\end{equation}
where $s \ge 0$, $\alpha, \beta, \gamma \in [1, + \infty]$, $p_2, \mu_2 \in (1, +\infty]$ such that
\[
\frac1\gamma + \frac1{p_2} + \frac1{\mu_2} = 1, \quad \frac1\alpha + \frac1\beta = 1.
\]
\end{lemma}

\begin{proof}
 Let $D^s \coloneqq (- \Delta)^{\frac s2}$. In the following all the estimates have to be thought component by component and then summed back together. From generalized fractional Leibniz rule (see e.g. \cite{G014} Thm. 1),
\begin{equation}
\begin{split}
\Big \lVert \big (\omega \ast  \langle \beta u,u \rangle \big ) u \Big \lVert_{\dot H^s} & = \Big \lVert  D ^s \big [ \big ( \omega \ast  \langle \beta u,u \rangle \big ) u \big] \Big \lVert_{L^2} \\
& \lesssim \Big  \lVert D^s \big ( \omega \ast  \langle \beta u,u \rangle \big ) \Big \lVert_{L^{p_1}} \lVert u \rVert_{L^{p_2}} + \big \lVert \omega \ast  \langle \beta u,u \rangle \big \rVert_{L^\infty}  \big \lVert D^s u \big \rVert_{L^2},
\end{split}
\end{equation}
where
\[
\frac 12 = \frac 1{p_1} + \frac1{p_2}, \quad p_1 \in (1, + \infty), \, p_2 \in (1, + \infty]
\]
and $p_1 \in [1,+ \infty)$ if $s=0$.
We estimate separately:\\
\begin{itemize}
\item
Using Young's inequality
\[
\big \lVert \omega \ast  \langle \beta u,u \rangle \big \rVert_{L^\infty} \lesssim \lVert \omega \rVert_{L^{\alpha}} \big \lVert \langle \beta u,u \rangle \big \rVert_{L^{\beta}} \lesssim \lVert \omega \rVert_{L^{\alpha}} \lVert u \rVert_{L^{2\beta}}^2,
\]
where 
\[
1 = \frac1{\alpha} + \frac1{\beta}, \quad \alpha, \beta \in [1, \infty];
\]
\item Using Young's inequality and generalized Leibniz rule
\[
\begin{split}
\Big  \lVert D^s \big ( \omega \ast  \langle \beta u,u \rangle \big ) \Big \lVert_{L^{p_1}} & = \Big  \lVert \omega \ast D^s \langle \beta u,u \rangle \Big \lVert_{L^{p_1}} \lesssim \lVert \omega \rVert_{L^{\gamma}} \big \lVert D^s \langle \beta u,u \rangle \big \rVert_{L^{\mu}} \\
& \lesssim \lVert \omega \rVert_{L^{\gamma}} \big \lVert D^s u \big \rVert_{L^2} \lVert u \rVert_{L^{\mu_2}},
\end{split}
\]
where
\[
1 + \frac1{p_1} = \frac 1{\gamma} + \frac1{\mu}, \qquad \frac1{\mu} = \frac12 + \frac1{\mu_2}, \qquad \mu, \gamma \in [1, +\infty], \, \mu_2 \in (1, \infty].
\]
Summing up, we obtain the result.
\end{itemize}
\end{proof}

\begin{lemma}
\label{lemma:H_s-difference}
Let $u,v \in \mathcal S(\R^3;\C^4)$, then 
\[
\begin{split}
 \big \lVert D^s \big (\langle \beta u,u \rangle - \langle \beta v, v \rangle \big ) \big \rVert_{L^{\mu}} & \lesssim \big \lVert D ^s ( u - v) \big \rVert_{L^{\mu_1}} ( \lVert u \rVert_{L^{\mu_2}} + \lVert v \rVert_{L^{\mu_2}} )+ \\
 & + \big ( \big \lVert D^s  u \big \rVert_{L^{\gamma_1}} + \big \lVert D^s  v \big \rVert_{L^{\gamma_1}} \big ) \lVert u-v \rVert_{L^{\gamma_2}}
\end{split}
\]
where $s \ge 0$ and 
\[
\frac 1\mu = \frac 1\mu_1 + \frac1\mu_2 = \frac 1\gamma_1 + \frac1\gamma_2.
\]
\end{lemma}

\begin{proof}
We notice that, if $u,v$ are scalar functions, we have
\[
D ^s ( \lvert u \rvert^2 - \lvert v \rvert^2 ) = D^s ( u \bar u - v \bar v \pm v \bar u) = D^s ( (u - v) \bar u ) + D^s ( v ( \bar u - \bar v) )
\]
Then we conclude arguing component by component and by the fractional Leibniz rule.
\end{proof}

We can now prove \Cref{thm:nla1}.
\begin{proof}
Let $T, M > 0$, we define the space
\[
X_{T,M} \coloneqq \{ u \in L^\infty_T H^s  \colon \lVert u \rVert_{L^\infty_T \dot H^s} \le M \},
\]
that, endowed with the metric 
\[
d(u,v) = \lVert u - v \rVert_{L^\infty_T \dot H^{ s}},
\]
is a complete metric space, and the map 
\begin{equation}
\label{map}
\Phi (u)(t,x) \coloneqq e^{i t \D_{\nu}} u_0(x) - i \int_0^t e^{i (t-s) \D_{\nu}} N (u) (s,x) ds.
\end{equation}
Then, the proof relies on standard contraction argument; we want to find $T,M$ such that $\Phi : X_{T,M} \rightarrow X_{T,M}$ is a contraction map on $(X_{T,M},d)$. \\
\emph{Step 1}: 
We observe that, since $s < \tfrac32$, $\dot H^s \hookrightarrow L^{\frac{6}{3 - 2s}}$ and that \Cref{lemma:Nu_Hs} holds with the choice
\[
\begin{cases}
\mu_2 = p_2 = 2 \beta = \frac6{3-2s} = 2p',\\
\alpha = \gamma = \frac3{2s} = p.
\end{cases}
\]
Moreover, since $s \le 1$ and $\lvert \nu \rvert < \frac{\sqrt3}2$, $\lVert u \rVert_{\dot H^s} \simeq \lVert u \rVert_{\dot H_{\D_{\nu}}^s}$ (\Cref{prop:equi_norm_3d}). Then, from Minkowski's inequality and \eqref{Nu_Hs}, we have
\[
\begin{split}
\big \lVert \Phi (u) \big \rVert_{L^\infty_T \dot H^s} & \lesssim \lVert u_0 \rVert_{\dot H^s} + \int_0^T \big \lVert  e^{i (t-\tau) \D_{\nu}} N(u)(\tau,x) \big \rVert_{L^\infty_T \dot H^s} d\tau \\
& \lesssim \lVert u_0 \rVert_{\dot H^s} + \int_0^T \big \lVert N(u) \big \rVert_{\dot H^s} d\tau\\
& \lesssim \lVert u_0 \rVert_{\dot H^s} + \lVert \omega \rVert_{L^p} \int_0^T \lVert u \rVert_{\dot H^s} \lVert u \rVert_{L^{\frac{6}{3 - 2s}}}^2 d\tau\end{split}
\]
and, from the Sobolev's embedding, we get
\[
\big \lVert \Phi (u) \big \rVert_{L^\infty_T \dot H^s} \lesssim \lVert u_0 \rVert_{\dot H^s} + \lVert \omega \rVert_{L^p} T \lVert u \rVert_{L^\infty_T \dot H^s}^3 \le C \lVert u_0 \rVert_{\dot H^s} + C \lVert \omega \rVert_{L^p} T M^3,
\]
for some $C >0$.\\
\emph{Step 2}: Proceeding as before, we get
\[
\begin{split}
\big \lVert \Phi(u) - \Phi(v) \big \rVert_{L^\infty_T \dot H^s} & \lesssim \int_0^T \big \lVert N(u) - N(v) \big \rVert_{\dot H^s} d\tau \\
&\lesssim \int_0^T \Big [ \big \lVert \big ( \omega \ast \big ( \langle \beta u,u \rangle - \langle \beta v, v \rangle \big ) \big )u \big \rVert_{\dot H^s} + \big \lVert \big ( \omega \ast \langle \beta v, v \rangle \big ) (u-v) \big \rVert_{\dot H^s} \Big ] d\tau. \\
& = \int_0^T (I + II)(\tau) \, d\tau.
\end{split}
\]
We estimate separately $I$ and $II$; from \Cref{lemma:Nu_Hs} and Young's inequality, we have
\[
\begin{split}
I & \lesssim \big \lVert \omega \ast \langle D \rangle^s \big (\langle \beta u,u \rangle - \langle \beta v, v \rangle \big ) \big \rVert_{L^{\frac3s}} \lVert u \rVert_{L^{\frac6{3-2s}}} + \big \lVert \omega \ast \big ( \langle \beta u,u \rangle - \langle \beta v, v \rangle \big ) \big \rVert_{L^\infty}\lVert u \rVert_{\dot H^s} \\
& \lesssim \lVert \omega \rVert_{L^p} \big \lVert \langle D \rangle^s \big (\langle \beta u,u \rangle - \langle \beta v, v \rangle \big ) \big \rVert_{L^{\frac3{3-s}}} \lVert u \rVert_{L^{\frac6{3-2s}}}+ \lVert \omega \rVert_{L^p} \lVert u - v \rVert_{L^{\frac6{3-2s}}} \big \lVert \lvert u \rvert + \lvert v \rvert \big \rVert_{L^{\frac6{3-2s}}} \lVert u \rVert_{\dot H^s}
\end{split}
\]
where for the second term we have used that, for two scalars functions $f,g$, $\big \lvert \lvert f \rvert^2 - \lvert g \rvert^2 \big \rvert \le \lvert f-g \rvert (\lvert f \rvert + \lvert g \vert)$.\footnote{Then, for $u,v$ is true that $\big \lvert \langle \beta u, u \rangle - \langle \beta v,v \rangle \big \rvert \le 4 \lvert u-v \rvert (\lvert u \rvert + \lvert v \rvert)$.}.
To estimate the first term we observe that \Cref{lemma:H_s-difference} holds taking
\[
\begin{cases}
\mu_1 = \gamma_1 =2,\\
\mu_2 = \gamma_2 = \frac6{3-2s},
\end{cases}
\]
then by the Sobolev's embedding, we have
\[
I \lesssim \lVert \omega \rVert_{L^p} \lVert u -v \rVert_{\dot H^s}  \lVert u \rVert_{\dot H^s} \big ( \lVert u \rVert_{\dot H^s} + \lVert v \rVert_{\dot H^s} \big ) 
\]
In addition, we have 
\[
\begin{split}
II & \lesssim \big \lVert \omega \ast \langle D \rangle^s \langle \beta v, v \rangle \big \rVert_{L^{\frac3s}} \lVert u-v \rVert_{L^{\frac6{3-2s}}} + \big \lVert \omega \ast \langle \beta v, v \rangle \big \rVert_{L^\infty} \lVert u-v \rVert_{\dot H^s} \\
& \lesssim \lVert \omega \rVert_{L^p} \lVert v \rVert_{L^{\frac6{3-2s}}} \lVert v \rVert_{\dot H^s} \lVert u-v \rVert_{L^{\frac6{3-2s}}} + \lVert \omega \rVert_{L^p} \lVert v \rVert_{L^{\frac6{3-2s}}}^2 \lVert u-v \rVert_{\dot H^s}.
\end{split}
\]
Putting all the estimates together, we have obtained that there exists $\tilde C >0$ such that
\[
\big \lVert \Phi(u) - \Phi(v) \big \rVert_{L^\infty_T \dot H^s} \le \tilde C \lVert \omega \rVert_{L^p}  \lVert u-v \rVert_{L^\infty_T \dot H^s} \int_0^T \big ( \lVert u \rVert_{\dot H^s} + \lVert v \rVert_{\dot H^s} \big) ^2 d\tau.
\]
Then, if we choose $T,M$ such that
\[
\begin{cases}
C \lVert u_0 \rVert_{\dot H^s} \le \frac{M}2,\\
T < \min \big \{ \frac{1}{2C \lVert \omega \rVert_{L^p} M^2}, \frac1{8 \tilde C \lVert \omega \rVert_{L^p} M^2} \big \},
\end{cases}
\]
$\Phi$ maps $X_{T,M}$ into itself and it is a contraction on $(X_{T,M}, d)$ and applying the Banach fixed-point theorem we get the claim.
\end{proof}

Now we turn to the proof of Theorem \ref{thm:nla2}; notice that in the proof above we didn't use any kind of Strichartz estimates. We will exploit them in the following.

\begin{proof}
The idea is to apply the Banach contraction theorem on a ball in the space $Y_T = L^\infty \big ([0,T]; \dot H^s (\R^3)\big ) \cap L^r([0,T]; L^q (\R^3))$, $T >0$, endowed with the norm
\[
\lVert u \rVert_{Y_T} \coloneqq  \sup_{t \in [0,T]} \lVert u \rVert_{\dot H^s(\R^3)} + \lVert u \rVert_{L^r([0,T], L^q(\R^3))},
\]
for some $(r,q)$ admissibile for \Cref{prop:stri_rad} so that
\[
\lVert e^{it \D_{\nu}} u_0 \rVert_{L^r_t L^q_x} \le C \lVert u_0 \rVert_{\dot H^s},
\]
where $s = \frac32 -\frac1r - \frac3q$ is such that $s \le 1$. We choose $(r,q)=(2p, 2p')$, where $p'$ is the conjugate exponent of $p$. We claim that this is an admissible couple and that we can find $T,M > 0$ such that
 the map $\Phi$ is as in \eqref{map}, is a contraction on the complete metric space $(Y_{T,M}, d)$, where 
\[
Y_{T,M} \coloneqq \big \{ u \in L^\infty_T \dot H^s \cap L^{2p}_T L^{2p'} \colon \lVert u \rVert_{Y_T} \le M \big \}, \quad d(u,v) \coloneqq \lVert u-v \rVert_{Y_T}.
\]
Let us now prove the claim.  We notice that if $p \in \big ( \frac{3}{1 + 2\sqrt{1-\nu^2}}, + \infty)$ then $(2p, 2p')$ satisfies conditions \eqref{eq:adm_cond_Dir-rad} and $s = \frac32 - \frac1{2p} - \frac3{2p'} = \frac1p \le 1$. Moreover, estimate \eqref{Nu_Hs} holds with the following choice of indexes
\[
\begin{cases}
\mu_2 = p_2 = 2\beta = 2p',\\
\alpha=\gamma=p,
\end{cases}
\]
Then, proceeding as before, we get
\[
\begin{split}
\big \lVert \Phi (u) \big \rVert_{\dot H^s} & \lesssim \lVert u_0 \rVert_{\dot H^s} + \int_0^T \big \lVert e^{i(t-\tau) \D_{\nu}} N(u) (\tau,x) \big \rVert_{L^{\infty}_T \dot H^s} d\tau \\
 & \lesssim \lVert u_0 \rVert_{\dot H^s} + \int_0^T \lVert N(u) \rVert_{\dot H^s} d\tau \\
& \lesssim \lVert u_0 \rVert_{\dot H^s} + \lVert \omega \rVert_{L^p} \lVert u \rVert_{L^{\infty}_T \dot H^s} \int_0^T \lVert u(\tau) \rVert_{L^{2p'}}^2 d\tau \\
& \lesssim \lVert u_0 \rVert_{\dot H^s} + T^{1 - \frac1p} \lVert \omega \rVert_{L^p} \lVert u \rVert_{L^{\infty}_T \dot H^s} \lVert u \rVert_{L^{2p}_T L^{2p'}} ^2\\
\end{split}
\]
and 
\[
\big \lVert \Phi(u) \big \rVert_{L^{2p}_T L^{2p'}} \lesssim \lVert u_0 \rVert_{\dot H^s} + T^{1 - \frac1p} \lVert \omega \rVert_{L^p} \lVert u \rVert_{L^\infty_T \dot H^s} \lVert u \rVert_{L^{2p}_T L^{2p'}}^2.
\]
Moreover
\[
\begin{split}
\big \lVert \Phi(u) - \Phi(v) \big \rVert_{Y_T}&\lesssim \int_0^T \big \lVert N(u) - N(v) \big \rVert_{\dot H^s} d\tau \\
& \lesssim \int_0^T \Big [ \big \lVert \big ( \omega \ast \big ( \langle \beta u,u \rangle - \langle \beta v, v \rangle \big ) \big )u \big \rVert_{\dot H^s} + \big \lVert \big ( \omega \ast \langle \beta v, v \rangle \big ) (u-v) \big \rVert_{\dot H^s} \Big ] d\tau \\
& = \int_0^T (I+II) (\tau) \, d\tau.
\end{split}
\]
We estimate $I, II$ separately; from \Cref{lemma:Nu_Hs} and Young's inequality, we have
\[
\begin{split}
I & \lesssim \big \lVert \omega \ast  D^s \big (\langle \beta u,u \rangle - \langle \beta v, v \rangle \big ) \big \rVert_{L^{2p}} \lVert u \rVert_{L^{2p'}} + \big \lVert \omega \ast \big ( \langle \beta u,u \rangle - \langle \beta v, v \rangle \big ) \big \rVert_{L^\infty}\lVert u \rVert_{\dot H^s} \\
& \lesssim \lVert \omega \rVert_{L^p} \big \lVert D^s \big ( \langle \beta u, u \rangle - \langle \beta v,v \rangle \big) \big \rVert_{L^{\mu}}\lVert u \rVert_{L^{2p'}} + \lVert \omega \rVert_{L^p} \big \lVert \langle \beta u,u \rangle - \langle \beta v, v \rangle \big \rVert_{L^{p'}} \lVert u \rVert_{\dot H^s},
\end{split}
\]
where $\mu = (2p)'$. From \Cref{lemma:H_s-difference}, we can estimate the $L^\mu$ norm as
\[
\big \lVert D^s \big ( \langle \beta u, u \rangle - \langle \beta v,v \rangle \big) \big \rVert_{L^{\mu}} \lesssim \lVert u- v \rVert_{\dot H^s} (\lVert u \rVert_{L^{2p'}} + \lVert v \rVert_{L^{2p'}} ) + \lVert u - v \rVert_{L^{2p'}} (\lVert u \rVert_{\dot H^s} + \lVert v \rVert_{\dot H^s}).
\]
Then, we can continue the chain of inequalities
\[
\begin{split}
I & \lesssim \lVert \omega \rVert_{L^p} \lVert u-v \rVert_{L^\infty_T \dot H^s}  (\lVert u \rVert_{L^{2p'}}  + \lVert v \rVert_{L^{2p'}} ) \lVert u \rVert_{L^{2p'}} + \\
& + \lVert \omega \rVert_{L^p} \lVert u-v \rVert_{L^{2p'}} \big [ \big ( \lVert u \rVert_{L^\infty_T \dot H^s} + \lVert v \rVert_{L^\infty_T \dot H^s})  \lVert u \rVert_{L^{2p'}} + \lVert u \rVert_{L^\infty_T \dot H^s} \lVert \lvert u \rvert + \lvert v \rvert \rVert_{L^{2p'}} \big]
\end{split}
\]
We estimate $II$ using again Young's inequality and fractional Leibniz rule, getting
\[
II \le \lVert \omega \rVert_{L^p} \lVert u-v\rVert_{L^\infty_T \dot H^s} \lVert v \rVert_{L^{2p'}}^2 + \lVert \omega \rVert_{L^p} \lVert u- v\rVert_{L^{2p'}} \lVert u \rVert_{L^\infty_T \dot H^s} \lVert v \rVert_{L^{2p'}}.
\]
Summing up, there exist two positive constants $C, \tilde C$ such that
\begin{gather}
\notag \big \lVert \Phi (u) \big \rVert_{Y_T} \le C \lVert u_0 \rVert_{Y_T} + C T^{1-\frac1p} \lVert \omega \rVert_{L^p} \lVert u \rVert^3_{L_T^{2p} L^{2p'}}, \\
\notag \big \lVert \Phi(u) - \Phi(v) \big \rVert_{Y_T} \le \tilde C T^{1 - \frac1p} \lVert \omega \rVert_{L^p}(\lVert u \rVert_{Y_T} + \lVert v \rVert_{Y_T}) ^2 \lVert u- v \rVert_{Y_T}.
\end{gather}
To conclude, if we choose $T,M > 0$ such that 
\[
\begin{cases}
C \lVert u_0 \rVert_{\dot H^s} \le \frac M2,\\
T^{1 - \frac1p} < \min \big \{ \frac1{2 C \lVert \omega \rVert_{L^p} M^2}, \frac1{8 \tilde C \lVert \omega \rVert_{L^p} M^2} \big \},
\end{cases}
\]
we get the claim.
\end{proof}

\medskip
{\bf Acknowledgments.} I wish to thank prof.\ Federico Cacciafesta for bringing this problem to my attention and for the constant guidance on my work. I am also grateful to prof.\ Anne-Sophie de Suzzoni for providing helpful comments and to the CMLS, École Polytechnique where I did part of this work. I was supported by the SS Chern Young Faculty Award funded by AX. Moreover, I am a member of GNAMPA-INDAM and I was partially supported by the INDAM-GNAMPA Project 2023 ``Analisi spettrale per equazioni di Dirac ed applicazioni'' CUPE53C22001930001.

\end{document}